\documentclass{amsart}
\usepackage{amssymb, amsmath, amsthm, graphics, comment, xspace, enumerate}
\usepackage{hyperref}
\newtheorem{thm}{Theorem}[section]

\newtheorem{lem}[thm]{Lemma}
\newtheorem{prop}[thm]{Proposition}
\theoremstyle{definition}
\newtheorem{defn}[thm]{Definition}
\newtheorem{example}[thm]{Example}
\theoremstyle{remark}
\newtheorem{rem}[thm]{Remark}
\numberwithin{equation}{section}
\begin{document}
\title[Generalized $c$-almost periodic type functions in ${\mathbb R}^{n}$]{Generalized $c$-almost periodic type functions in ${\mathbb R}^{n}$}

\author{M. Kosti\' c}
\address{Faculty of Technical Sciences,
University of Novi Sad,
Trg D. Obradovi\' ca 6, 21125 Novi Sad, Serbia}
\email{marco.s@verat.net}

{\renewcommand{\thefootnote}{} \footnote{
\\ \text{  }  \ \    {\it Key words and phrases.} Quasi-asymptotically $c$-almost periodic type functions, $(S,{\mathbb D})$-asymptotically 
$(\omega,c)$-periodic type functions, $S$-asymptotically $(\omega_{j},c_{j},{\mathbb D}_{j})_{j\in {\mathbb N}_{n}}$-periodic type functions, semi-$(c_{j})_{j\in {\mathbb N}_{n}}$-periodic type functions, Weyl $c$-almost periodic type functions,
abstract Volterra integro-differential equations.
\\  \text{  }  
Marko Kosti\' c is partially supported by grant 451-03-68/2020/14/200156 of Ministry
of Science and Technological Development, Republic of Serbia.}}

\begin{abstract}
In this paper, we analyze multi-dimensional quasi-asymptotically $c$-almost periodic functions and their Stepanov generalizations as well as multi-dimensional Weyl $c$-almost periodic type functions.
We also analyze several important subclasses of the class of multi-dimensional quasi-asymptotically $c$-almost periodic functions and reconsider the notion of semi-$c$-periodicity in the multi-dimensional setting, working in the general framework of Lebesgue spaces with variable exponent. We
provide
certain applications of our results to
the abstract Volterra integro-differential equations in Banach spaces.  
\end{abstract}
\maketitle

\section{Introduction and preliminaries}

The notion of almost periodicity was introduced by the Danish mathematician H. Bohr around 1924-1926 and later reconsidered by many others. Suppose that $I$ is either ${\mathbb R}$ or $[0,\infty)$ and $f : I \rightarrow X$ is a given continuous function, where $X$ is a complex Banach space equipped with the norm $\| \cdot \|$. If $\varepsilon>0,$ then we say that a positive real number $\tau>0$ is a $\varepsilon$-period for $f(\cdot)$ if and only if\index{$\varepsilon$-period}
$
\| f(t+\tau)-f(t) \| \leq \varepsilon,$ $ t\in I.
$
The set constituted of all $\varepsilon$-periods for $f(\cdot)$ is denoted by $\vartheta(f,\varepsilon).$ We say that the function $f(\cdot)$ is almost periodic if and only if for each $\varepsilon>0$ the set $\vartheta(f,\varepsilon)$ is relatively dense in $[0,\infty),$ which means that
there exists a finite real number $l>0$ such that any subinterval of $[0,\infty)$ of length $l$ meets $\vartheta(f,\varepsilon)$. For more details about almost periodic functions and their applications, we refer the reader to \cite{besik, diagana, fink, gaston, nova-mono, 188, pankov, 30}.

The class of $S$-asymptotically $\omega$-periodic functions, where $\omega>0,$ was introduced by
H. R. Henr\'iquez, M. Pierri and P. T\' aboas in \cite{pierro}. This class of continuous functions has different ergodicity properties
compared with the classes of $\omega$-periodic functions and asymptotically $\omega$-periodic functions, and it is not so easily comparable with the class of almost periodic functions since an $S$-asymptotically $\omega$-periodic function is not necessarily uniformly continuous. For some applications of 
$S$-asymptotically $\omega$-periodic functions, we refer the reader to \cite{cuevas-souza, dimbour, guengai, hrh}. 

In \cite{brazil}, we have recently analyzed the class of quasi-asymptotically almost periodic functions. Any
$S$-asymptotically $\omega$-periodic function 
$f: I \rightarrow X$ is quasi-asymptotically almost periodic, while the converse statement is not true in general. The class of Stepanov $p$-quasi-asymptotically almost periodic functions, which has been also analyzed in \cite{brazil}, contains all asymptotically Stepanov $p$-almost periodic functions and make a subclass of the class consisting of all Weyl $p$-almost periodic functions in the sense of general approach of A. S. Kovanko \cite{kovanko}; thus, in \cite{brazil}, we have actually initiated the study of generalized (asymptotical) almost periodicity that intermediate the Stepanov concept and a very general Weyl concept.

The main purpose of research articles \cite{marko-manuel-ap}-\cite{stmarko-manuel-ap}, written in a collaboration with A. Ch\'avez, K. Khalil and M. Pinto, was to analyze
various classes of (Stepanov) almost periodic functions of form $F : \Lambda \times X\rightarrow Y,$ where $(Y,\|\cdot \|_{Y})$ is a complex Banach spaces and $\emptyset \neq  \Lambda \subseteq {\mathbb R}^{n}$. 
In our recent joint research article \cite{rmjm} with V. Fedorov, we have continued the research studies \cite{marko-manuel-ap}-\cite{stmarko-manuel-ap} by developing the basic theory of multi-dimensional Weyl almost periodic type functions in Lebesgue spaces with variable exponents (see also the research articles 
\cite{marko-manuel-aa}-\cite{stmarko-manuel-aa} and \cite{multi-weyl} for further information concerning multi-dimensional almost automorphic type functions as well as their Stepanov and Weyl generalizations).

On the other hand, the notion of $(\omega,c)$-periodicity and various generalizations of this concept have recently been introduced and analyzed by E. Alvarez, A. G\'omez, M. Pinto \cite{alvarez1},
E. Alvarez, S. Castillo, M. Pinto \cite{alvarez2}-\cite{alvarez3} and M. Fe\v{c}kan, K. Liu, J.-R. Wang \cite{Reference040}. In  our joint research article \cite{c1} with M. T. Khalladi, A. Rahmani, M. Pinto and D. Velinov,
we have investigated 
$c$-almost periodic type functions and their applications (the notion of $c$-almost periodicity, depending only on the parameter $c,$ is substantially different from the notion of $(\omega,c)$-periodicity and the recently analyzed notion of $(\omega,c)$-almost periodicity; see the forthcoming research monograph \cite{nova-selected} for more details about the subject). The analysis 
from \cite{c1} has been continued in \cite{BIMV}, where the same group of authors has analyzed Weyl $c$-almost periodic type functions, quasi-asymptotically $c$-almost periodic type functions and $S$-asymptotically $(\omega ,c)$-periodic type functions in the one-dimensional setting, as well as in the research articles \cite{multi-ce} and \cite{multi-omega-ce}, where the author of this paper has analyzed multi-dimensional $c$-almost periodic type functions and various classes of multi-dimensional $(\omega,c)$-almost periodic type functions. 

The main aim of this paper is to continue the research studies raised in the above-mentioned papers by introducing and investigating various classes of multi-dimensional quasi-asymptotically $c$-almost periodic functions, multi-dimensional semi-$c$-periodic functions, multi-dimensional Weyl $c$-almost periodic functions (see the article \cite{andres1} by J. Andres and D. Pennequin for the initial study of semi-periodicity as well as \cite{Chaouchi}, \cite{JM} for more deatils about this topic) and their applications to the abstract Volterra integro-differential equations.

The organization of paper can be briefly described as follows. After recalling the basic definitions and facts about asymptotically $c$-almost periodic functions in the multi-dimensional framework, we remind the readers of the basic definitions and facts about Lebesgue spaces with variable exponents
$L^{p(x)}$ (Subsection \ref{karambita}), almost periodic type functions in
${\mathbb R}^{n}$ (Subsection \ref{stavisub}),
$({\bf \omega},c)$-periodic functions and $({\bf \omega}_{j},c_{j})_{j\in {\mathbb N}_{n}}$-periodic functions (Subsection \ref{krizni-stab}). Following our approach from \cite{c1}-\cite{JM} and \cite{multi-omega-ce}, in Section \ref{profica-jazz} we introduce and analyze $(S,{\mathbb D})$-asymptotically 
$(\omega,c)$-periodic type functions, $S$-asymptotically $(\omega_{j},c_{j},{\mathbb D}_{j})_{j\in {\mathbb N}_{n}}$-periodic type functions and semi-$(c_{j},{\mathcal B})_{j\in {\mathbb N}_{n}}$-periodic functions (the last class of functions is investigated in Subsection \ref{semi-cejot}); here, it is worth noting that the notion of $(S,{\mathbb D})$-asymptotical
$(\omega,c)$-periodicity seems to be new even in the one-dimensional setting.
Various classes of multi-dimensional quasi-asymptotically $c$-almost periodic functions are examined in Section \ref{profica-efg} following the approach obeyed in \cite{BIMV} and \cite{dumath2}, while the Stepanov generalizations of multi-dimensional quasi-asymptotically $c$-almost periodic type functions are examined in Section \ref{zlatnidecko} (the introduced classes seem to be new and not considered elsewhere even in the case that the exponent $p(\cdot)$ has a constant value). 
The main aim of Section \ref{multiWeylce} is to continue our analysis of Weyl $c$-almost periodic type functions from \cite{BIMV} in the   multi-dimensional setting. Some applications of our results to the abstract Volterra integro-differential equations are presented in Section \ref{apply}; we also provide numerous illustrative examples henceforth. 

We use the standard notation throughout the paper. By $(X,\| \cdot \|)$ and
$(Y,\| \cdot \|_{Y})$ we denote two complex Banach spaces.  By
$L(X,Y)$ we denote the Banach algebra of all bounded linear operators from $X$ into
$Y$ with $L(X,X)$ being denoted $L(X)$. 
The convolution product $\ast$ of measurable functions $f: {\mathbb R}^{n} \rightarrow {\mathbb C}$ and $g: {\mathbb R}^{n} \rightarrow X$ is defined by $(f\ast g)({\bf t}):=\int_{{\mathbb R}^{n}}f({\bf t}-{\bf s})g({\bf s})\,
d{\bf s},$ ${\bf t}\in {\mathbb R}^{n},$ whenever the limit exists; $\langle \cdot, \cdot \rangle$ denotes the usual inner product in ${\mathbb R}^{n}.$ The shorthand
$\chi_{A}(\cdot)$ denotes the characteristic function of a set $A\subseteq {\mathbb R}^{n}.$ If ${\bf t_{0}}\in {\mathbb R}^{n}$ and $\epsilon>0$, then we define $B({\bf t}_{0},\epsilon):=\{{\bf t } \in {\mathbb R}^{n} : |{\bf t}-{\bf t_{0}}| \leq \epsilon\},$
where $|\cdot|$ denotes the Euclidean norm in ${\mathbb R}^{n};$ by $(e_{1},e_{2},\cdot \cdot \cdot,e_{n})$ we denote the standard basis of ${\mathbb R}^{n}.$ Set 
${\mathbb N}_{n}:=\{1,\cdot \cdot \cdot, n\}.$
We will always assume henceforth that ${\mathcal B}$ is a collection of non-empty subsets of $X$ such that, for  every $x\in X,$ there exists $B\in{\mathcal B}$ with $x\in B.$ 

\subsection{Lebesgue spaces with variable exponents
$L^{p(x)}$}\label{karambita}

The basic reference about the Lebesgue spaces with variable exponents
$L^{p(x)}$ is the research monograph \cite{variable} by L. Diening, P. Harjulehto, P. H\"ast\"uso and M. Ruzicka. 

Suppose that $\emptyset \neq \Omega \subseteq {\mathbb R}^{n}$ is a non-empty Lebesgue measurable subset and  
$M(\Omega  : X)$ denotes the collection of all measurable functions $f: \Omega \rightarrow X;$ $M(\Omega):=M(\Omega : {\mathbb R}).$ By ${\mathcal P}(\Omega)$ we denote the vector space of all Lebesgue measurable functions $p : \Omega \rightarrow [1,\infty].$
For any $p\in {\mathcal P}(\Omega)$ and $f\in M(\Omega : X),$ we set
$$
\varphi_{p(x)}(t):=\left\{
\begin{array}{l}
t^{p(x)},\quad t\geq 0,\ \ 1\leq p(x)<\infty,\\ \\
0,\quad 0\leq t\leq 1,\ \ p(x)=\infty,\\ \\
\infty,\quad t>1,\ \ p(x)=\infty 
\end{array}
\right.
$$
and
$$
\rho(f):=\int_{\Omega}\varphi_{p(x)}(\|f(x)\|)\, dx .
$$
We define the Lebesgue space 
$L^{p(x)}(\Omega : X)$ with variable exponent
by
$$
L^{p(x)}(\Omega : X):=\Bigl\{f\in M(\Omega : X): \lim_{\lambda \rightarrow 0+}\rho(\lambda f)=0\Bigr\}.
$$
Equivalently,
\begin{align*}
L^{p(x)}(\Omega : X)=\Bigl\{f\in M(\Omega : X):  \mbox{ there exists }\lambda>0\mbox{ such that }\rho(\lambda f)<\infty\Bigr\};
\end{align*}
see, e.g., \cite[p. 73]{variable}.
For every $u\in L^{p(x)}(\Omega : X),$ we introduce the Luxemburg norm of $u(\cdot)$ by
$$
\|u\|_{p(x)}:=\|u\|_{L^{p(x)}(\Omega :X)}:=\inf\Bigl\{ \lambda>0 : \rho(u/\lambda)    \leq 1\Bigr\}.
$$
Equipped with the above norm, the space $
L^{p(x)}(\Omega : X)$ becomes a Banach space (see e.g. \cite[Theorem 3.2.7]{variable} for the scalar-valued case), coinciding with the usual Lebesgue space $L^{p}(\Omega : X)$ in the case that $p(x)=p\geq 1$ is a constant function.
If $p\in M(\Omega),$ we define
$$
p^{-}:=\text{essinf}_{x\in \Omega}p(x) \ \ \mbox{ and } \ \ p^{+}:=\text{esssup}_{x\in \Omega}p(x).
$$
Set
$$
D_{+}(\Omega ):=\bigl\{ p\in M(\Omega): 1 \leq p^{-}\leq p(x) \leq p^{+} <\infty \mbox{ for a.e. }x\in \Omega \bigr \}.
$$
In the case that $p\in D_{+}(\Omega),$ the space $
L^{p(x)}(\Omega : X)$ behaves nicely, with almost all fundamental properties of the Lesbesgue space with constant exponent $
L^{p}(\Omega : X)$ being retained; in this case, 
$$
L^{p(x)}(\Omega : X)=\Bigl\{f\in M(\Omega : X)  \, ; \,  \mbox{ for all }\lambda>0\mbox{ we have }\rho(\lambda f)<\infty\Bigr\}.
$$

We will use the following lemma (cf. \cite{variable} for the scalar-valued case):

\begin{lem}\label{aux}
\begin{itemize}
\item[(i)] (The H\"older inequality) Let $p,\ q,\ r \in {\mathcal P}(\Omega)$ such that
$$
\frac{1}{q(x)}=\frac{1}{p(x)}+\frac{1}{r(x)},\quad x\in \Omega .
$$
Then, for every $u\in L^{p(x)}(\Omega : X)$ and $v\in L^{r(x)}(\Omega),$ we have $uv\in L^{q(x)}(\Omega : X)$
and
\begin{align*}
\|uv\|_{q(x)}\leq 2 \|u\|_{p(x)}\|v\|_{r(x)}.
\end{align*}
\item[(ii)] Let $\Omega $ be of a finite Lebesgue's measure and let $p,\ q \in {\mathcal P}(\Omega)$ such $q\leq p$ a.e. on $\Omega.$ Then
 $L^{p(x)}(\Omega : X)$ is continuously embedded in $L^{q(x)}(\Omega : X),$ and the constant of embedding is less than or equal to 
$2(1+m(\Omega)).$
\item[(iii)] Let $f\in L^{p(x)}(\Omega : X),$ $g\in M(\Omega : X)$ and $0\leq \|g\| \leq \|f\|$ a.e. on $\Omega .$ Then $g\in L^{p(x)}(\Omega : X)$ and $\|g\|_{p(x)}\leq \|f\|_{p(x)}.$
\item[(iv)] Suppose that $f\in L^{p(x)}(\Omega : X)$ and $A\in L(X,Y).$ 
Then $Af \in L^{p(x)}(\Omega : Y)$ and 
$\|Af\|_{L^{p(x)}(\Omega : Y)}\leq \|A\| \cdot \|f\|_{L^{p(x)}(\Omega : X)}.$
\end{itemize}
\end{lem}

For further information concerning the Lebesgue spaces with variable exponents
$L^{p(x)},$ we refer the reader to \cite{variable}, \cite{fan-zhao} and \cite{doktor}.  See also \cite{toka-mbape}-
\cite{toka-mbape-prim} and references cited therein. 

\subsection{Almost periodic type functions in ${\mathbb R}^{n}$}\label{stavisub}

Suppose that $F : {\mathbb R}^{n} \rightarrow X$ is a continuous function. Then we say that $F(\cdot)$ is almost periodic if and only if for each $\epsilon>0$
there exists $l>0$ such that for each ${\bf t}_{0} \in {\mathbb R}^{n}$ there exists ${\bf \tau} \in B({\bf t}_{0},l)$ with
\begin{align*}
\bigl\|F({\bf t}+{\bf \tau})-F({\bf t})\bigr\| \leq \epsilon,\quad {\bf t}\in {\mathbb R}^{n}.
\end{align*}
This is equivalent to saying that for any sequence $({\bf b}_n)$ in ${\mathbb R}^{n}$ there exists a subsequence $({\bf a}_{n})$ of $({\bf b}_n)$
such that $(F(\cdot+{\bf a}_{n}))$ converges in $C_{b}({\mathbb R}^{n}: X),$ the Banach space of bounded continuous functions $F :{\mathbb R}^{n}\rightarrow X$ equipeed with the sup-norm. Any trigonometric polynomial in ${\mathbb R}^{n}$ is almost periodic and it is also well known that $F(\cdot)$ is almost periodic if and only if there exists a sequence of trigonometric polynomials in ${\mathbb R}^{n}$ which converges uniformly to $F(\cdot);$
here, by a trigonometric polynomial in ${\mathbb R}^{n}$ we mean any linear combination of functions like ${\bf t}\mapsto e^{i\langle
 {\bf \lambda}, {\bf t} \rangle},$ ${\bf t}\in{\mathbb R}^{n},$ where ${\bf \lambda}\in {\mathbb R}^{n}.$ 
Any almost periodic function $F: {\mathbb R}^{n} \rightarrow X$ is almost periodic with respect to each of the variables but the converse statement is not true in general. Further on,
any almost periodic function $F(\cdot)$ is bounded, uniformly continuous and the mean value
$$
M(F):=\lim_{T\rightarrow +\infty}\frac{1}{T^{n}}\int_{{\bf s}+K_{T}}F({\bf t})\, d{\bf t}
$$
exists and it does not depend on $s\in [0,\infty)^{n};$ here,  
$K_{T}:=\{ {\bf t}=(t_{1},t_{2},\cdot \cdot \cdot,t_{n}) \in {\mathbb R}^{n} :  0\leq t_{i}\leq T\mbox{ for }1\leq i\leq n\}.$ The Bohr-Fourier coefficient $F_{\lambda}\in X$ is defined by \index{Bohr-Fourier coefficient}
$$
F_{\lambda}:=M\Bigl(e^{-i\langle \lambda, {\bf \cdot}\rangle }F(\cdot)\Bigr),\quad \lambda \in {\mathbb R}^{n}.
$$
The Bohr spectrum of $F(\cdot),$ defined by\index{Bohr spectrum}
$$
\sigma(F):=\bigl\{ \lambda \in {\mathbb R}^{n} : F_{\lambda}\neq 0 \bigr\},
$$
is at most a countable set.

If $F : {\mathbb R}^{n} \rightarrow X$ is an almost periodic function, then $F(\cdot)$ is uniformly recurrent, i.e., $F(\cdot)$ is continuous and there exists a sequence $({\bf \tau}_{k})$ in ${\mathbb R}^{n}$ such that $\lim_{k\rightarrow +\infty} |{\bf \tau}_{k}|=+\infty$ and
$$
\lim_{k\rightarrow +\infty}\sup_{{\bf t}\in {\mathbb R}^{n}} \bigl\|F({\bf t}+{\bf \tau}_{k})-F({\bf t})\bigr\| =0.
$$
We say that a function $F : {\mathbb R}^{n} \rightarrow X$ is asymptotically uniformly recurrent if and only if there exist a uniformly recurrent function $G : {\mathbb R}^{n} \rightarrow X$ and a function $Q\in C_{0}({\mathbb R}^{n} : X)$ such that $F({\bf t})=G({\bf t})+Q({\bf t})$ for all ${\bf t}\in {\mathbb R}^{n};$ here, $C_{0}({\mathbb R}^{n} : X)$ denotes the vector space of continuous functions vanishing at zero
when $|{\bf t}| \rightarrow +\infty.$

We need the following definitions from \cite{marko-manuel-ap} and \cite{multi-ce}:

\begin{defn}\label{nafaks123456789012345123}
Suppose that 
${\mathbb D} \subseteq I \subseteq {\mathbb R}^{n},$ $c\in {\mathbb C} \setminus \{0\}$ and the set ${\mathbb D}$ is unbounded, as well as
$\emptyset  \neq I'\subseteq I \subseteq {\mathbb R}^{n},$ $F : I \times X \rightarrow Y$ is a continuous function and $I +I' \subseteq I.$ Then we say that
$F(\cdot;\cdot)$ is ${\mathbb D}$-asymptotically Bohr $({\mathcal B},I',c)$-almost periodic  of type $1$ if and only if for every $B\in {\mathcal B}$ and $\epsilon>0$
there exist $l>0$ and $M>0$ such that for each ${\bf t}_{0} \in I'$ there exists ${\bf \tau} \in B({\bf t}_{0},l) \cap I'$ such that
\begin{align*}
\bigl\|F({\bf t}+{\bf \tau};x)-c F({\bf t};x)\bigr\|_{Y} \leq \epsilon,\mbox{ provided } {\bf t},\ {\bf t}+\tau \in {\mathbb D}_{M},\ x\in B.
\end{align*}
\end{defn}

\begin{defn}\label{kompleks12345}
Suppose that 
${\mathbb D} \subseteq I \subseteq {\mathbb R}^{n}$ and the set ${\mathbb D}$  is unbounded. By $C_{0,{\mathbb D},{\mathcal B}}(I \times X :Y)$ we denote the vector space consisting of all continuous functions $Q : I \times X \rightarrow Y$ such that, for every $B\in {\mathcal B},$ we have $\lim_{t\in {\mathbb D},|t|\rightarrow +\infty}Q({\bf t};x)=0,$ uniformly for $x\in B.$ For any $T>0,$ we set ${\mathbb D}_{T}:=\{ {\bf t}\in {\mathbb D} : |{\bf t}|\geq T \}.$ 
\end{defn}

In our further analyses of Stepanov and Weyl classes, the regions $I$ and $I'$ will be also denoted by $\Lambda$ and $\Lambda',$ respectively (we aim to stay consistent with the notation used in \cite{nova-selected}).

\begin{defn}\label{nafaks-vstepanov}
Suppose that $\omega \subseteq {\mathbb R}^{n}$ is a Lebesgue measurable set with positive Lebesgue measure,
${\mathbb D} \subseteq \Lambda \subseteq {\mathbb R}^{n}$ and the set ${\mathbb D}$ is unbounded, as well as
$\emptyset  \neq  \Lambda'\subseteq  \Lambda \subseteq {\mathbb R}^{n},$ $F :  \Lambda \times X \rightarrow Y$ is a continuous function and $ \Lambda + \Lambda' \subseteq \Lambda.$ Then we say that:
\begin{itemize}
\item[(i)]\index{function!${\mathbb D}$-asymptotically Stepanov $(\Omega,p({\bf u}))$-$({\mathcal B}, \Lambda')$-almost periodic of type $1$}
$F(\cdot;\cdot)$ is Stepanov $(\Omega,p({\bf u}))$-$({\mathcal B}, \Lambda')$-almost periodic of type $1$ if and only if for every $B\in {\mathcal B}$ and $\epsilon>0$
there exist $l>0$ and $M>0$ such that for each ${\bf t}_{0} \in \Lambda'$ there exists ${\bf \tau} \in B({\bf t}_{0},l) \cap \Lambda'$ such that
\begin{align*}
\bigl\|F({\bf t}+{\bf \tau}+{\bf u};x)-F({\bf t}+{\bf u};x)\bigr\|_{L^{p({\bf u})}(\Omega :Y)} \leq \epsilon,\mbox{ provided } {\bf t},\ {\bf t}+\tau \in {\mathbb D}_{M},\ x\in B.
\end{align*}
\item[(ii)] \index{function!${\mathbb D}$-asymptotically Stepanov $(\Omega,p({\bf u}))$-$({\mathcal B}, \Lambda')$-uniformly recurrent  of type $1$}
$F(\cdot;\cdot)$ is ${\mathbb D}$-asymptotically Stepanov $(\Omega,p({\bf u}))$-$({\mathcal B}, \Lambda')$-uniformly recurrent of type $1$ if and only if for every $B\in {\mathcal B}$ 
there exist a sequence $({\bf \tau}_{k})$ in $\Lambda'$ and a sequence $(M_{k})$ in $(0,\infty)$ such that $\lim_{k\rightarrow +\infty} |{\bf \tau}_{k}|=\lim_{k\rightarrow +\infty}M_{k}=+\infty$ and
$$
\lim_{k\rightarrow +\infty}\sup_{{\bf t},{\bf t}+{\bf \tau}_{k}\in {\mathbb D}_{M_{k}};x\in B} \bigl\|F({\bf t}+{\bf \tau}_{k}+{\bf u};x)-F({\bf t}+{\bf u};x)\bigr\|_{L^{p({\bf u})}(\Omega :Y)} =0.
$$
\end{itemize}
If $\Lambda'=\Lambda,$ then we also say that
$F(\cdot;\cdot)$ is ${\mathbb D}$-asymptotically Stepanov $(\Omega,p({\bf u}))$-${\mathcal B}$-almost periodic of type $1$ (${\mathbb D}$-asymptotically Stepanov $(\Omega,p({\bf u}))$-${\mathcal B}$-uniformly recurrent  of type $1$); furthermore, if $X\in {\mathcal B},$ then it is also said that $F(\cdot;\cdot)$ is ${\mathbb D}$-asymptotically Stepanov $(\Omega,p({\bf u}))$-$\Lambda'$-almost periodic  of type $1$ (${\mathbb D}$-asymptotically Stepanov $\Lambda'$-uniformly recurrent  of type $1$). If $\Lambda'=\Lambda$ and $X\in {\mathcal B}$, then we also say that $F(\cdot;\cdot)$ is ${\mathbb D}$-asymptotically Stepanov almost periodic  of type $1$ (${\mathbb D}$-asymptotically Stepanov uniformly recurrent of type $1$). Here and hereafter we will  remove the prefix ``${\mathbb D}$-'' in the case that ${\mathbb D}=\Lambda$ and remove the prefix ``$({\mathcal B},)$''  in the case that $X\in {\mathcal B}.$ 
\end{defn}

\subsection{$({\bf \omega},c)$-Periodic functions and $({\bf \omega}_{j},c_{j})_{j\in {\mathbb N}_{n}}$-periodic functions}\label{krizni-stab}

A
continuous
function $F:I\rightarrow X$ is said to be Bloch $({\bf p},{\bf k})$-periodic, or Bloch
periodic with period ${\bf p}$ and Bloch wave vector or Floquet exponent ${\bf k},$ where ${\bf p}\in {\mathbb R}^{n}$ and ${\bf k}\in {\mathbb R}^{n},$ if and only if 
$
F({\bf t}+{\bf p})=e^{i\langle {\bf k}, {\bf p}\rangle}F({\bf t}),$ ${\bf t}\in I
$ (we assume here that ${\bf p}+I \subseteq I$).
Following the recent research analyses of E. Alvarez, A. G\'omez, M. Pinto \cite{alvarez1} and
E. Alvarez, S. Castillo, M. Pinto \cite{alvarez2}-\cite{alvarez3}, we have recently extended the notion of Bloch $({\bf p},{\bf k})$-periodicity in the following way:

\begin{defn}\label{drasko-presing} (\cite{multi-omega-ce})
Let ${\bf \omega}\in {\mathbb R}^{n} \setminus \{0\},$ $c\in {\mathbb C} \setminus \{0\}$
and 
${\bf \omega}+I \subseteq I$. A continuous
function $F:I\rightarrow X$ is said to be $({\bf \omega},c)$-periodic if and only if 
$
F({\bf t}+{\bf \omega})=cF({\bf t}),$ ${\bf t}\in I.
$ 
\end{defn}

If $F:I\rightarrow X$ is a Bloch $({\bf p},{\bf k})$-periodic function, then $F(\cdot)$ is $({\bf p},c)$-periodic with $c=e^{i\langle {\bf k}, {\bf p}\rangle};$ conversely, if $|c|=1$  and 
$F:I\rightarrow X$ is $({\bf \omega},c)$-periodic, then we can always find a point ${\bf k}\in {\mathbb R}^{n}
$ such that the function $F(\cdot)$ is Bloch $({\bf p},{\bf k})$-periodic. If $c=1,$ resp. $c=-1,$ then we say that  the function $F(\cdot)$ is $\omega$-periodic, resp. $\omega$-anti-periodic.
If $|c|\neq 1$ 
and $F:I\rightarrow X$ is $({\bf \omega},c)$-periodic, then $
F({\bf t}+m{\bf \omega})=c^{m}F({\bf t}),$ ${\bf t}\in I,
$ $m\in {\mathbb N},$ so that the existence of a point ${\bf t}_{0}\in I$ such that $F({\bf t}_{0})\neq 0$ implies $\lim_{m\rightarrow \infty}||F({\bf t}_{0}+m{\bf \omega})||=+\infty,$ provided that $|c|>1$, and 
$\lim_{m\rightarrow \infty}||F({\bf t}_{0}+m{\bf \omega})||=0,$ provided that $|c|<1.$

\begin{defn}\label{drasko-presing1} (\cite{multi-omega-ce})
Let ${\bf \omega}_{j}\in {\mathbb R} \setminus \{0\},$ $c_{j}\in {\mathbb C} \setminus \{0\}$
and 
${\bf \omega}_{j}e_{j}+I \subseteq I$ ($1\leq j\leq n$). A continuous
function $F:I\rightarrow X$ is said to be $({\bf \omega}_{j},c_{j})_{j\in {\mathbb N}_{n}}$-periodic if and only if 
$
F({\bf t}+{\bf \omega}_{j}e_{j})=c_{j}F({\bf t}),$ ${\bf t}\in I,
$ $j\in {\mathbb N}_{n}.$ 
\end{defn} \index{function!$({\bf \omega}_{j},c_{j})_{j\in {\mathbb N}_{n}}$-periodic}

If $c_{j}=1$ for all $j\in {\mathbb N}_{n},$ resp. $c_{j}=-1$ for all $j\in {\mathbb N}_{n},$ then we also say that  the function $F(\cdot)$ is $(\omega_{j})_{j\in {\mathbb N}_{n}}$-periodic, resp. $(\omega_{j})_{j\in {\mathbb N}_{n}}$-anti-periodic. 
The classes of $({\bf \omega},c)$-periodic functions and $({\bf \omega}_{j},c_{j})_{j\in {\mathbb N}_{n}}$-periodic functions are closed under the operation of the pointwise convergence of functions. 

Let $c_{j}\in {\mathbb C} \setminus \{0\}$. Then 
it is said that a continuous function $F:I\rightarrow X$ is $(c_{j})_{j\in {\mathbb N}_{n}}$-periodic if and only if there exist real numbers
${\bf \omega}_{j}\in {\mathbb R} \setminus \{0\}$ such that 
${\bf \omega}_{j}e_{j}+I \subseteq I$ ($1\leq j\leq n$) and the function
$F:I\rightarrow X$ is $({\bf \omega}_{j},c_{j})_{j\in {\mathbb N}_{n}}$-periodic. 
It can be simply verified that the assumption $|c_{j}|=1$ for all $j\in {\mathbb N}_{n}$ implies that any $(c_{j})_{j\in {\mathbb N}_{n}}$-periodic function $F: {\mathbb R}^{n}\rightarrow X$ is
almost periodic. 

In \cite{multi-omega-ce}, we have also introduced the following notion: 

\begin{defn}\label{kakavsam jadebil}
Suppose that ${\mathbb D} \subseteq I \subseteq {\mathbb R}^{n},$ the set ${\mathbb D}$  is unbounded,
${\bf \omega}\in {\mathbb R}^{n} \setminus \{0\},$ $c\in {\mathbb C} \setminus \{0\},$
${\bf \omega}+I \subseteq I,$
${\bf \omega}_{j}\in {\mathbb R} \setminus \{0\},$ $c_{j}\in {\mathbb C} \setminus \{0\},$ 
${\bf \omega}_{j}e_{j}+I \subseteq I$ ($1\leq j\leq n$) and $F : I \times X \rightarrow Y.$ 
Then we say that the function $F(\cdot; \cdot)$ is $({\mathbb D},{\mathcal B})$-asymptotically $({\bf \omega},c)$-periodic, resp. 
$({\mathbb D},{\mathcal B})$-asymptotically $({\bf \omega}_{j},c_{j})_{j\in {\mathbb N}_{n}}$-periodic, if and only if there exist
an $({\bf \omega},c)$-periodic, resp. an $({\bf \omega}_{j},c_{j})_{j\in {\mathbb N}_{n}}$-periodic, function 
$F_{0} : I \times X \rightarrow Y$ (by that we mean that for each fixed element $x\in X$ the function $F(\cdot ;x)$ is $({\bf \omega},c)$-periodic, resp. $({\bf \omega}_{j},c_{j})_{j\in {\mathbb N}_{n}}$-periodic) and a function $Q\in C_{0,{\mathbb D},{\mathcal B}}(I :X)$ such that 
$F({\bf t};x)=F_{0}({\bf t};x)+Q({\bf t};x),$ ${\bf t} \in I,$ $x\in X.$
\end{defn}

Before we proceed to our next section, we would like to note that the notions of $(\omega,c)$-periodicity and $(\omega_{j},c_{j})_{j\in {\mathbb N}_{n}}$-periodicity have 
been generalized in several other directions \cite{multi-omega-ce}; for example, in this paper, we have considered several various classes of
$({\bf \omega}_{j},c_{j}; r_{j}, {\mathbb I}_{j}')_{j\in {\mathbb N}_{n}}$-almost periodic type functions. We will not deal with these classes of functions henceforth.

\section{$(S,{\mathbb D},{\mathcal B})$-asymptotically 
$(\omega,c)$-periodic type functions, $(S,{\mathcal B})$-asymptotically $(\omega_{j},c_{j},{\mathbb D}_{j})_{j\in {\mathbb N}_{n}}$-periodic type functions and semi-$(c_{j},{\mathcal B})_{j\in {\mathbb N}_{n}}$-periodic type functions}\label{profica-jazz}

This section investigates the classes of $(S,{\mathbb D})$-asymptotically
$(\omega,c)$-periodic type functions, $S$-asymptotically $(\omega_{j},c_{j},{\mathbb D}_{j})_{j\in {\mathbb N}_{n}}$-periodic type functions and\\ semi-$(c_{j},{\mathcal B})_{j\in {\mathbb N}_{n}}$-periodic type functions.
In the following two definitions, we extend the recently introduced notion of $S_{c}$-asymptotical periodicity (cf. M. T. Khalladi, M. Kosti\' c, M. Pinto, A. Rahmani and D. Velinov \cite[Definition 3.1]{BIMV}, where the authors have considered the case in which  $X=\{0\}$ and
$I={\mathbb D}={\mathbb D}_{1}$ is ${\mathbb R}$ or $[0,\infty)$)
and its subnotions: the $S$-asymptotical Bloch $(\omega,c)$-periodicity, resp. $S$-asymptotical $\omega$-anti-periodicity (see \cite[Definition 3.1, Definition 3.2]{y-k chang}, where Y.-K. Chang and Y. Wei have considered the particular cases $|c|=1,$ resp. $c=-1,$ $X=\{0\}$ and
$I={\mathbb R}={\mathbb D}={\mathbb D}_{1}$): 

\begin{defn}\label{drasko-presing-12345}
Let ${\bf \omega}\in {\mathbb R}^{n} \setminus \{0\},$ $c\in {\mathbb C} \setminus \{0\},$
${\bf \omega}+I \subseteq I,$ ${\mathbb D} \subseteq I \subseteq {\mathbb R}^{n}$ and the set ${\mathbb D}$ be unbounded. A continuous
function $F:I \times X\rightarrow Y$ is said to be $(S,{\mathbb D},{\mathcal B})$-asymptotically 
$(\omega,c)$-periodic if and only if for each $B\in {\mathcal B}$ we have
\begin{align*}
\lim_{|{\bf t}|\rightarrow +\infty,{\bf t}\in {\mathbb D}}\bigl\| F({\bf t}+\omega ; x)-cF({\bf t} ; x)\bigr\|_{Y}=0,\quad \mbox{ uniformly in }x\in B.
\end{align*}
\end{defn}

\begin{defn}\label{drasko-presing123456}
Let ${\bf \omega}_{j}\in {\mathbb R} \setminus \{0\},$ $c_{j}\in {\mathbb C} \setminus \{0\},$
${\bf \omega}_{j}e_{j}+I \subseteq I$,
${\mathbb D}_{j} \subseteq I \subseteq {\mathbb R}^{n}$ and the set ${\mathbb D}_{j}$ be unbounded ($1\leq j\leq n$).
A continuous
function $F:I \times X\rightarrow Y$ is said to be  $(S,{\mathcal B})$-asymptotically $({\bf \omega}_{j},c_{j},{\mathbb D}_{j})_{j\in {\mathbb N}_{n}}$-periodic if and only if for each $ j\in {\mathbb N}_{n}$ we have
\begin{align*}
\lim_{|{\bf t}|\rightarrow +\infty,{\bf t}\in {\mathbb D}_{j}}\bigl\|F({\bf t}+{\bf \omega}_{j}e_{j};x)-c_{j}F({\bf t};x)\bigr\|_{Y}=0,\quad \mbox{ uniformly in }x\in B.
\end{align*}
\end{defn} 

Before going any further, we will present an illustrative example:

\begin{example}\label{primerchic}
Let $X:=c_{0}({\mathbb C})$ 
be the Banach space of all numerical sequences tending to zero, equipped with the sup-norm. Suppose that $\omega_{j}=2\pi,$
$c_{j}\in {\mathbb C} $ and $|c_{j}|=1$ for all $j\in {\mathbb N}_{n}.$ From \cite[Example 2.12]{multi-omega-ce}, we know that the function 
$$
F_{1}\bigl(t_{1},\cdot \cdot \cdot,t_{n}\bigr):=\prod_{j=1}^{n}c_{j}^{\frac{t_{j}}{2\pi}}\sin t_{j},\quad {\bf t}=(t_{1},\cdot \cdot \cdot,t_{n})\in [0,\infty)^{n}
$$ 
is $(2\pi,c_{j})_{j\in {\mathbb N}_{n}}$-periodic. On the other hand, from \cite[Example 3.1]{pierro} and \cite[Example 2.6]{brazil}, we know that the function 
$$
f(t):=\Biggl(\frac{4k^{2}t^{2}}{(t^{2}+k^{2})^{2}} \Biggr)_{k\in {\mathbb N}},\ t\geq 0
$$
is $S$-asymptotically $\omega$-periodic for any positive real number $\omega>0,$ as well as that its range
is not relatively compact in $X$ and $f(\cdot)$ is uniformly continuous; let us only note here that R. Xie and C. Zhang
have constructed, in \cite[Example 17]{xie}, an example of an
$S$-asymptotically $\omega$-periodic function which is not uniformly continuous. Set
$$
F\bigl(t_{1},\cdot \cdot \cdot,t_{n},t_{n+1}\bigr):=F_{1}\bigl(t_{1},\cdot \cdot \cdot,t_{n}\bigr) \cdot f\bigl(t_{n+1}\bigr),\quad \bigl(t_{1},\cdot \cdot \cdot,t_{n},t_{n+1}\bigr) \in [0,\infty)^{n+1}.
$$
Then the function $F(\cdot)$ is $S$-asymptotically $({\bf \omega}_{j},c_{j},{\mathbb D}_{j})_{j\in {\mathbb N}_{n+1}}$-periodic, where $c_{n+1}=1,$ $\omega_{n+1}>0$ being arbitrary, ${\mathbb D}_{j}=[0,\infty)^{n+1}$ for $1\leq j\leq n$ and 
${\mathbb D}_{n+1}=K\times [0,\infty)$ ($\emptyset \neq K \subseteq [0,\infty)^{n}$ is a compact set), as easily approved. See also \cite[Example 2.16, Example 2.17, Example 2.18]{brazil}.
\end{example}

Immediately from the corresponding definitions, we have the following result:

\begin{prop}\label{nulti-jax}
\begin{itemize}
\item[(i)] Let ${\bf \omega}\in {\mathbb R}^{n} \setminus \{0\},$ $c\in {\mathbb C} \setminus \{0\},$
${\bf \omega}+I \subseteq I,$ ${\mathbb D} \subseteq I \subseteq {\mathbb R}^{n}$ and the set ${\mathbb D}$ be unbounded. If $\omega +{\mathbb D} \subseteq {\mathbb D}$ and the function $F: I\times X \rightarrow Y$ is 
$({\mathbb D},{\mathcal B})$-asymptotically $({\bf \omega},c)$-periodic, then 
the function $F(\cdot; \cdot)$ is 
$(S,{\mathbb D},{\mathcal B})$-asymptotically 
$(\omega,c)$-periodic.
\item[(ii)] Let ${\bf \omega}_{j}\in {\mathbb R} \setminus \{0\},$ $c_{j}\in {\mathbb C} \setminus \{0\},$
${\bf \omega}_{j}e_{j}+I \subseteq I$,
${\mathbb D}_{j} \subseteq I \subseteq {\mathbb R}^{n}$ and the set ${\mathbb D}_{j}$ be unbounded ($1\leq j\leq n$).
If $\omega e_{j} +{\mathbb D} \subseteq {\mathbb D}$ and the function $F: I\times X\rightarrow Y$ is 
$({\mathbb D},{\mathcal B})$-asymptotically $({\bf \omega}_{j},c_{j})_{j\in {\mathbb N}_{n}}$-periodic, then the function $F(\cdot ; \cdot)$ is 
$(S,{\mathcal B})$-asymptotically $({\bf \omega}_{j},c_{j},{\mathbb D}_{j})_{j\in {\mathbb N}_{n}}$-periodic with ${\mathbb D}_{j}\equiv {\mathbb D}$ for all $j\in {\mathbb N}_{n}.$
\end{itemize}
\end{prop}

We will provide the proof of the first part of the following simple result 
for the sake of completeness:

\begin{prop}\label{piksi}
\begin{itemize}
\item[(i)] Let ${\bf \omega}\in {\mathbb R}^{n} \setminus \{0\},$ $c\in {\mathbb C} \setminus \{0\},$ $\omega +I\subseteq I,$
${\mathbb D} \subseteq I \subseteq {\mathbb R}^{n}$ and the set ${\mathbb D}$ be unbounded. If for each $B\in {\mathcal B}$ there exists $\epsilon_{B}>0$ such that
the sequence $(F_{k}(\cdot ;\cdot))$ of $(S,{\mathbb D},{\mathcal B})$-asymptotically 
$(\omega,c)$-periodic functions 
converges uniformly to a function $F(\cdot ;\cdot)$ on the set $B^{\circ} \cup \bigcup_{x\in \partial B}B(x,\epsilon_{B}),$ then $F(\cdot;\cdot)$ is $(S,{\mathbb D},{\mathcal B})$-asymptotically 
$(\omega,c)$-periodic.
\item[(ii)] Let ${\bf \omega}_{j}\in {\mathbb R} \setminus \{0\},$ $c_{j}\in {\mathbb C} \setminus \{0\},$
${\bf \omega}_{j}e_{j}+I \subseteq I$,
${\mathbb D}_{j} \subseteq I \subseteq {\mathbb R}^{n}$ and the set ${\mathbb D}_{j}$ be unbounded ($1\leq j\leq n$). 
If for each $B\in {\mathcal B}$ there exists $\epsilon_{B}>0$ such that
the sequence $(F_{k}(\cdot ;\cdot))$ of $(S,{\mathcal B})$-asymptotically $({\bf \omega}_{j},c_{j},{\mathbb D}_{j})_{j\in {\mathbb N}_{n}}$-periodic functions converges uniformly to a function $F(\cdot ;\cdot)$ on the set $B^{\circ} \cup \bigcup_{x\in \partial B}B(x,\epsilon_{B}),$ then the function $F(\cdot;\cdot)$ is $(S,{\mathcal B})$-asymptotically\\ $({\bf \omega}_{j},c_{j},{\mathbb D}_{j})_{j\in {\mathbb N}_{n}}$-periodic.
\end{itemize}
\end{prop}

\begin{proof}
The validity of (i) can be deduced as follows. By the proofs of \cite[Proposition 2.7, Proposition 2.8]{marko-manuel-ap}, it follows that the function $F(\cdot;\cdot)$ is continuous. Let $\epsilon>0$ and $B\in {\mathcal B}$ be fixed. Then there exists $k_{0}\in {\mathbb N}$ such that $\| F_{k_{0}}({\bf t};x)-F({\bf t};x)\|_{Y} \leq \epsilon/3(1+|c|)$ for all $({\bf t},x) \in I \times B.$ Further on, there exists $M>0$ such that the assumptions $|{\bf t}|>M,$ ${\bf t}\in {\mathbb D}$ and $x\in B$ imply $\| F_{k_{0}}({\bf t}+\omega ; x)-cF_{k_{0}}({\bf t} ; x)\|_{Y}<\epsilon /3.$ Then the final conclusion follows from the well known decomposition and estimates
\begin{align*}
& \bigl\|F({\bf t}+\omega ; x)-cF({\bf t} ; x) \bigr\|_{Y}
\\& \leq  \bigl\|F({\bf t}+\omega ; x)-F_{k_{0}}({\bf t} ; x) \bigr\|_{Y}
+\bigl\| F_{k_{0}}({\bf t}+\omega ; x)-cF_{k_{0}}({\bf t} ; x)\bigr\|_{Y}
\\& +|c| \cdot \bigl\|F_{k_{0}}({\bf t}+\omega ; x)-cF_{k_{0}}({\bf t} ; x) \bigr\|_{Y}\leq 3\cdot (\epsilon/3)=\epsilon.
\end{align*}
\end{proof}

The convolution invariance of function spaces introduced in Definition \ref{drasko-presing-12345} and Definition \ref{drasko-presing123456} can be shown under very mild assumptions:

\begin{thm}\label{milenko}
Suppose that $h\in L^{1}({\mathbb R}^{n})$ and $F : {\mathbb R}^{n} \times X \rightarrow Y$ is a continuous function satisfying that for each $B\in {\mathcal B}$ there exists a finite real number $\epsilon_{B}>0$ such that
$\sup_{{\bf t}\in {\mathbb R}^{n},x\in B^{\cdot}}\|F({\bf t},x)\|_{Y}<+\infty,$
where $B^{\cdot} \equiv B^{\circ} \cup \bigcup_{x\in \partial B}B(x,\epsilon_{B}).$ 
\begin{itemize}
\item[(i)] Suppose that ${\mathbb D}={\mathbb R}^{n}.$ Then the function
\begin{align}\label{gariprekrsaj}
(h\ast F)({\bf t};x):=\int_{{\mathbb R}^{n}}h(\sigma) F({\bf t}-\sigma;x)\, d\sigma,\quad {\bf t}\in {\mathbb R}^{n},\ x\in X 
\end{align}
is well defined and for each $B\in {\mathcal B}$ we have $\sup_{{\bf t}\in {\mathbb R}^{n},x\in B^{\cdot}}\|(h\ast F)({\bf t};x)\|_{Y}<+\infty;$ furthermore, if $F(\cdot;\cdot)$ is $(S, {\mathbb R}^{n},{\mathcal B})$-asymptotically 
$(\omega,c)$-periodic, then the function $(h\ast F)(\cdot;\cdot)$ is $(S, {\mathbb R}^{n},{\mathcal B})$-asymptotically 
$(\omega,c)$-periodic.
\item[(ii)] Suppose that ${\mathbb D}_{j}={\mathbb R}^{n}$ for all $j\in {\mathbb N}_{n}.$ Then the function $(h\ast F)(\cdot;\cdot),$ given by \eqref{gariprekrsaj}, is well defined and for each $B\in {\mathcal B}$ we have $\sup_{{\bf t}\in {\mathbb R}^{n},x\in B^{\cdot}}\|(h\ast F)({\bf t};x)\|_{Y}<+\infty;$ moreover, if the function $F(\cdot;\cdot)$ is $(S,{\mathcal B})$-asymptotically\\ $({\bf \omega}_{j},c_{j},{\mathbb R}^{n})_{j\in {\mathbb N}_{n}}$-periodic, then the function $(h\ast F)(\cdot;\cdot)$ is likewise $(S,{\mathcal B})$-asymptotically $({\bf \omega}_{j},c_{j},{\mathbb R}^{n})_{j\in {\mathbb N}_{n}}$-periodic.
\end{itemize}
\end{thm}

\begin{proof}
We will prove only (i). It is clear that the function $(h\ast F)(\cdot;\cdot)$ is well defined as well as that $\sup_{{\bf t}\in {\mathbb R}^{n},x\in B^{\cdot}}\|(h\ast F)({\bf t};x)\|_{Y}<+\infty$ for all $B\in {\mathcal B}$. Its continuity at the fixed point $({\bf t}_{0};x_{0}) \in {\mathbb R}^{n} \times X$ follows from the existence of a set $B\in {\mathcal B}$ such that $x_{0}\in B,$ the assumption $\sup_{{\bf t}\in {\mathbb R}^{n},x\in B^{\cdot}}\|F({\bf t};x)\|_{Y}<+\infty$  and the dominated convergence theorem. Let $\epsilon>0$ and $B\in {\mathcal B}$ be fixed. Then there exists a sufficiently large real number $M>0$ such that $\| F({\bf t}+\omega ; x)-cF({\bf t};x)\|_{Y}<\epsilon/2,$ provided $|{\bf t}|>M_{1}$ and $x\in B.$
Therefore, there exists a finite constant $c_{B}\geq 1$ such that
\begin{align*}
\Bigl\| & (h\ast F)({\bf t}+\omega ; x)-c(h\ast F)({\bf t};x)\Bigr\|_{Y}
\\& \leq \int_{{\mathbb R}^{n}}|h(\sigma)| \cdot \bigl\| F({\bf t}+\omega -\sigma; x)-cF({\bf t}-\sigma;x)\bigr\|_{Y} \, d\sigma
\\& = \int_{|\sigma| \leq M_{1}}|h({\bf t}-\sigma)| \cdot \bigl\| F(\sigma+\omega ; x)-cF(\sigma;x)\bigr\|_{Y} \, d\sigma 
\\& +\int_{|\sigma| \geq M_{1}}|h({\bf t}-\sigma)| \cdot \bigl\| F(\sigma+\omega ; x)-cF(\sigma ;x)\bigr\|_{Y} \, d\sigma
\\& \leq \epsilon/2 +\int_{|\sigma| \geq M_{1}}|h({\bf t}-\sigma)| \cdot \bigl\| F(\sigma+\omega ; x)-cF(\sigma;x)\bigr\|_{Y} \, d\sigma
\\& \leq \epsilon/2 +c_{B}\int_{|\sigma| \geq M_{1}}|h({\bf t}-\sigma)|  \, d\sigma.
\end{align*}
On the other hand, there exists a finite real number $M_{2}>0$ such that 
\\ $\int_{|\sigma| \geq M_{2}}|h(\sigma)|\, d\sigma <\epsilon/2c_{B}.$ If $|{\bf t}|>M_{1}+M_{2},$ then for each $\sigma \in {\mathbb R}^{n}$ with $|\sigma| \leq M_{1}$ we have $|{\bf t}-\sigma|\geq M_{2}.$ 
This simply implies the required conclusion.
\end{proof}

The following result connects the notion introduced in Definition \ref{drasko-presing-12345} and Definition \ref{drasko-presing123456}:

\begin{prop}\label{drale}
Let ${\bf \omega}_{j}\in {\mathbb R} \setminus \{0\},$ $c_{j}\in {\mathbb C} \setminus \{0\},$
${\bf \omega}_{j}e_{j}+I \subseteq I,$ 
${\mathbb D}_{j} \subseteq I \subseteq {\mathbb R}^{n}$ and the set ${\mathbb D}_{j}$ be unbounded ($1\leq j\leq n$).
If $F:I \times X \rightarrow Y$ is $(S,{\mathcal B})$-asymptotically $({\bf \omega}_{j},c_{j},{\mathbb D}_{j})_{j\in {\mathbb N}_{n}}$-periodic and the set ${\mathbb D}$ consisting of all tuples ${\bf t}\in {\mathbb D}_{n}$ such that ${\bf t}+\sum_{i=j+1}^{n}\omega_{i}e_{i}$ for all $j\in {\mathbb N}_{n-1}$
is unbounded in ${\mathbb R}^{n}$, then the function $F(\cdot;\cdot)$ is $(S,{\mathbb D},{\mathcal B})$-asymptotically 
$(\omega,c)$-periodic, with $\omega:=\sum_{j=1}^{n}\omega_{j}e_{j} $ and
$c:=\prod_{j=1}^{n}c_{j}.$ 
\end{prop}

\begin{proof}
The proof simply follows from the corresponding definitions and the next estimates:
\begin{align*}
&\bigl\| F({\bf t}+\omega;x)-cF({\bf t};x)\bigr\|  = \Bigl\| F\bigl(t_{1}+\omega_{1},\cdot \cdot \cdot, t_{n}+\omega_{n};x \bigr)-c_{1}\cdot \cdot \cdot c_{n} F\bigl( t_{1},\cdot \cdot \cdot, t_{n};x \bigr)\Bigr\|
\\& \leq 
\Bigl\| F\bigl(t_{1}+\omega_{1},t_{2}+\omega_{2},\cdot \cdot \cdot, t_{n}+\omega_{n} ;x\bigr)-c_{1} F\bigl( t_{1},t_{2}+\omega_{2},\cdot \cdot \cdot, t_{n}+\omega_{n};x \bigr)\Bigr\|
\\& + \bigl| c_{1} \bigr| \cdot \Bigl\| F\bigl( t_{1},t_{2}+\omega_{2},\cdot \cdot \cdot, t_{n}+\omega_{n};x \bigr)-c_{2}\cdot \cdot \cdot c_{n}  F\bigl( t_{1},\cdot \cdot \cdot, t_{n} ;x\bigr)\Bigr\|
\\& \leq \Bigl\| F\bigl(t_{1}+\omega_{1},t_{2}+\omega_{2},\cdot \cdot \cdot, t_{n}+\omega_{n} ;x\bigr)-c_{1} F\bigl( t_{1},t_{2}+\omega_{2},\cdot \cdot \cdot, t_{n}+\omega_{n};x \bigr)\Bigr\|
\\& +\bigl| c_{1} \bigr| \cdot  \Biggl[  \Bigl\| F\bigl( t_{1},t_{2}+\omega_{2},\cdot \cdot \cdot, t_{n}+\omega_{n};x \bigr)-c_{2}F\bigl( t_{1},t_{2},\cdot \cdot \cdot, t_{n}+\omega_{n};x \bigr)\Bigr\|
\\&+\bigl| c_{2}\bigr| \cdot \Bigl\| F\bigl( t_{1},t_{2},\cdot \cdot \cdot, t_{n}+\omega_{n};x \bigr)-c_{3}\cdot \cdot \cdot c_{n} F\bigl( t_{1},t_{2},\cdot \cdot \cdot, t_{n};x \bigr)\Bigr\| \Biggr]
\\& \leq \cdot \cdot \cdot.
\end{align*}
\end{proof}

The proof of following proposition is simple and therefore omitted:

\begin{prop}\label{qomegaca}
Let ${\bf \omega},\ a\in {\mathbb R}^{n} \setminus \{0\},$ $c\in {\mathbb C} \setminus \{0\},$ $\alpha \in {\mathbb C},$
${\bf \omega}+I \subseteq I$
and $a+I \subseteq I.$
Suppose that the
functions $F:I \times X\rightarrow Y$ and  $G:I \times X\rightarrow Y$ are
$(S,{\mathbb D},{\mathcal B})$-asymptotically $(\omega,c)$-periodic ($(S,{\mathcal B})$-asymptotically $({\bf \omega}_{j},c_{j},{\mathbb D}_{j})_{j\in {\mathbb N}_{n}}$-periodic). Then we have the following:
\begin{itemize}
\item[(i)]
The function $\check{F}(\cdot; \cdot)$
is $(S,-{\mathbb D},{\mathcal B})$-asymptotically $(-\omega,c)$-periodic ($(S,{\mathcal B})$-asymptotically $(-{\bf \omega}_{j},c_{j},-{\mathbb D}_{j})_{j\in {\mathbb N}_{n}}$-periodic), where $\check{F}({\bf t} ;x):=F(-{\bf t};x),$ ${\bf t}\in -I,$ $x\in X.$ 
\item[(ii)]
The functions $\| F(\cdot ; \cdot)\|,$ $[F+G](\cdot ;\cdot)$ and $\alpha F(\cdot ;\cdot)$ are $(S,{\mathbb D},{\mathcal B})$-asymptotically $(\omega,|c|)$-periodic ($(S,{\mathcal B})$-asymptotically $({\bf \omega}_{j},|c_{j}|,{\mathbb D}_{j})_{j\in {\mathbb N}_{n}}$-periodic).
\item[(iii)]
If $a+{\mathbb D} \subseteq {\mathbb D}$ ($a+{\mathbb D}_{j} \subseteq {\mathbb D}_{j}$ for all $j\in {\mathbb N}_{n}$) and $y\in X$, then the function
$F_{a,y}: I \times X\rightarrow Y$ defined by
$F_{a,y}({\bf t};x):=F({\bf t}+a;x+y),$ ${\bf t}\in I,$ $x\in X$ is 
$(S,{\mathbb D},{\mathcal B}_{y})$-asymptotically $(\omega,c)$-periodic ($(S,{\mathcal B}_{y})$-asymptotically $({\bf \omega}_{j},c_{j},{\mathbb D}_{j})_{j\in {\mathbb N}_{n}}$-periodic), where ${\mathcal B}_{y}:=\{ -y+B : B\in {\mathcal B}\}.$
\item[(iv)] If ${\bf \omega}\in {\mathbb R}^{n} \setminus \{0\},$ $c_{i}\in {\mathbb C} \setminus \{0\}$ for $i=1,2,$
${\bf \omega}+I \subseteq I,$
the function $G:I \times X \rightarrow {\mathbb C}$ is $(S,{\mathbb D},{\mathcal B})$-asymptotically 
$(\omega,c_{1})$-periodic
and the function $H:I \times X\rightarrow Y$ is $(S,{\mathbb D},{\mathcal B})$-asymptotically 
$(\omega,c_{2})$-periodic, then the function $F(\cdot):=G(\cdot)H(\cdot)$ is $(S,{\mathbb D},{\mathcal B})$-asymptotically 
$(\omega,c_{1}c_{2})$-periodic, provided that for each set $B\in {\mathcal B}$ we have $\sup_{{\bf t}\in I; x\in B}[| G({\bf t};x) | + \| F({\bf t};x) \|_{Y}]<\infty.$
\item[(v)]
Let ${\bf \omega}_{j}\in {\mathbb R} \setminus \{0\},$ $c_{j,i}\in {\mathbb C} \setminus \{0\}$ and
${\bf \omega}_{j}e_{j}+I \subseteq I$ ($1\leq j\leq n,$ $1\leq i\leq 2$). Suppose that the
function $G : I \times X \rightarrow {\mathbb C}$ is 
$(S,{\mathcal B})$-asymptotically $({\bf \omega}_{j},c_{j,1},{\mathbb D}_{j})_{j\in {\mathbb N}_{n}}$-periodic
and the function $H : I \rightarrow X$ is 
$(S,{\mathcal B})$-asymptoti-\\cally $({\bf \omega}_{j},c_{j,2},{\mathbb D}_{j})_{j\in {\mathbb N}_{n}}$-periodic.
Set
$
c_{j}:=c_{j,1}c_{j,2}, $ $1\leq j\leq n.$
Then the function $F(\cdot):=G(\cdot)H(\cdot)$ is $(S,{\mathcal B})$-asymptotically $({\bf \omega}_{j},c_{j},{\mathbb D}_{j})_{j\in {\mathbb N}_{n}}$-periodic, provided that for each set $B\in {\mathcal B}$ we have $\sup_{{\bf t}\in I; x\in B}[| G({\bf t};x) | + \| F({\bf t};x) \|_{Y}]<\infty.$
\end{itemize}
\end{prop}

It should be noted that the classes of $({\bf \omega},c)$-periodic functions and $({\bf \omega}_{j},c_{j})_{j\in {\mathbb N}_{n}}$-periodic functions can be profiled in the following way (\cite{multi-omega-ce}; see also Example 2.18 of this paper for an interesting application):
\begin{itemize}
\item[(i)] Let ${\bf \omega}=(\omega_{1},\omega_{2},\cdot \cdot \cdot,\omega_{n})\in {\mathbb R}^{n} \setminus \{0\},$ 
${\bf \omega}+I \subseteq I$, $c\in {\mathbb C} \setminus \{0\}$ and $S:=\{i\in {\mathbb N}_{n} : \omega_{i}\neq 0 \}.$ Denote by ${\mathrm A}$ the collection of all tuples 
${\bf a}=(a_{1},a_{2},\cdot \cdot \cdot,a_{|S|})\in {\mathbb R}^{|S|}$ 
such that $\sum_{i\in S}a_{i}=1.$
Then a continuous
function $F:I\rightarrow X$ is $({\bf \omega},c)$-periodic
if and only if, for every (some) ${\bf a} \in {\mathrm A},$ the function $G_{{\bf a}} : I\rightarrow X$, defined by
\begin{align*}
G_{{\bf a}}\bigl(t_{1},t_{2},\cdot \cdot \cdot,t_{n}\bigr):=c^{-\sum_{i\in S}\frac{a_{i}t_{i}}{\omega_{i}}}F\bigl(t_{1},t_{2},\cdot \cdot \cdot,t_{n}\bigr),\quad {\bf t}=\bigl(t_{1},t_{2},\cdot \cdot \cdot,t_{n}\bigr)\in I,
\end{align*}
is $({\bf \omega},1)$-periodic.
\item[(ii)] Let ${\bf \omega}_{j}\in {\mathbb R} \setminus \{0\},$ $c_{j}\in {\mathbb C} \setminus \{0\},$
${\bf \omega}_{j}e_{j}+I \subseteq I$ ($1\leq j\leq n$) and the
function $F:I\rightarrow X$ is continuous. For each $j\in {\mathbb N}_{n},$ we define the function $G_{j} : I\rightarrow X$ by 
\begin{align*}
G_{j}\bigl(t_{1},t_{2},\cdot \cdot \cdot,t_{n}\bigr):=c_{j}^{-\frac{t_{j}}{\omega_{j}}}F\bigl(t_{1},t_{2},\cdot \cdot \cdot,t_{n}\bigr),\quad {\bf t}=\bigl(t_{1},t_{2},\cdot \cdot \cdot,t_{n}\bigr)\in I.
\end{align*}
Then $F(\cdot)$ is $({\bf \omega}_{j},c_{j})_{j\in {\mathbb N}_{n}}$-periodic if and only if, for every ${\bf t}=(t_{1},t_{2},\cdot \cdot \cdot,t_{n})\in I$ and $j\in {\mathbb N}_{n}$, we have
\begin{align*}
G_{j}\bigl(t_{1},t_{2},\cdot \cdot \cdot,t_{j}+\omega_{j}, \cdot \cdot \cdot, t_{n}\bigr)=G_{j}\bigl(t_{1},t_{2},\cdot \cdot \cdot,t_{j}, \cdot \cdot \cdot, t_{n}\bigr).
\end{align*}
\end{itemize}

Using these clarifications, we can introduce various spaces of pseudo-like $(S,{\mathbb D},{\mathcal B})$-asymptotically $(\omega,c)$-periodic type functions and pseudo-like $(S,{\mathcal B})$-asymptotically $({\bf \omega}_{j},c_{j},{\mathbb D}_{j})_{j\in {\mathbb N}_{n}}$-periodic type functions following the method proposed in 
\cite[Definition 2.4, Definition 2.5]{alvarez2} and
\cite[Definition 2.4, Definition 2.5]{alvarez3}; we will skip all related details for simplicity. The interested reader may also try to formulate extensions of \cite[Proposition 3.1, Corollary 3.1-Corollary 3.2]{BIMV} in the multi-dimensional setting.

\subsection{Semi-$(c_{j},{\mathcal B})_{j\in {\mathbb N}_{n}}$-periodic functions}\label{semi-cejot}

In this subsection, we will briefly exhibit the main results about the class of multi-dimensional semi-$(c_{j},{\mathcal B})_{j\in {\mathbb N}_{n}}$-periodic functions.
For the sake of brevity, we will always assume here that the region $I$ has the form $I=I_{1}\times I_{2} \times \cdot ... \times I_{n},$ where each set $I_{j}$ is equal to ${\mathbb R},$ $(-\infty,a_{j}]$ or $[a_{j},\infty)$ for some real number $a_{j}\in {\mathbb N}$ ($1\leq j\leq n$).

We will use the following definition:

\begin{defn}\label{kuckanjecejot}
Suppose that $ F :  I \times X \rightarrow Y$ is a continuous function and $c_{j}\in {\mathbb C} \setminus \{0\}$ ($1\leq j\leq n$). Then we say that $F(\cdot;\cdot)$ is semi-$(c_{j},{\mathcal B})_{j\in {\mathbb N}_{n}}$-periodic if and only if, 
for every $\epsilon>0$ and $B\in {\mathcal B},$
there exist real numbers $\omega_{j}\in {\mathbb R} \setminus \{0\}$
such that $\omega_{j}e_{j}+I\subseteq I$ ($1\leq j\leq n$) and
\begin{align}\label{semiceaf}
\Bigl\| F\bigl({\bf t}+m\omega_{j}e_{j};x\bigr)-c_{j}^{m}F({\bf t};x)\Bigr\|\leq \epsilon,\quad m\in {\mathbb N},\ j\in {\mathbb N}_{n}, \ {\bf t} \in {\mathbb R}^{n},\ x\in B.
\end{align}
The function $F(\cdot;\cdot)$ is said to be semi-${\mathcal B}$-periodic if and only if $F(\cdot;\cdot)$ is semi-$(c_{j},{\mathcal B})_{j\in {\mathbb N}_{n}}$-periodic with $c_{j}=1$ for all $j\in {\mathbb N}_{n}.$
\end{defn}

Suppose that $j\in {\mathbb N}_{n},$ $x\in X$ and $|c_{j}|\neq 1.$ Fix the variables $t_{1},\cdot\cdot\ \cdot, t_{j-1},t_{j+1},\cdot \cdot \cdot, t_{n}.$ Then there exist three possibilities:\vspace{0.1cm}

1. $I_{j}={\mathbb R}.$ Then, due to \eqref{semiceaf}, the function $f : {\mathbb R} \rightarrow Y$ given by $f(t):=F(t_{1},\cdot\cdot\ \cdot, t_{j-1},t, t_{j+1},\cdot \cdot \cdot, t_{n}),$ $t\in {\mathbb R}$ is semi-$c_{j}$-periodic of type $1_{+}$
in the sense of \cite[Definition 3(i)]{JM} and therefore $f(\cdot)$ is $c_{j}$-periodic due to \cite[Theorem 1]{JM}. Hence, the function 
$F(\cdot;x)$
is $c_{j}$-periodic in the variable $t_{j}.$

2. $I_{j}=[a_{j},+\infty)$ for some real number $a_{j} \in {\mathbb R}.$
Then the function 
$F(\cdot;x)$
is $c_{j}$-periodic in the variable $t_{j},$ which follows from the same argumentation applied to the function $f(t):=F(t_{1},\cdot\cdot \cdot, t_{j-1},t-a_{j}, t_{j+1},\cdot \cdot \cdot, t_{n}),$ $t\geq 0.$

3. $I_{j}=(-\infty,a_{j}]$ for some real number $a_{j} \in {\mathbb R}.$ 
Then the function 
$F(\cdot;x)$
is $c_{j}$-periodic in the variable $t_{j},$ which follows from the same argumentation applied to the function $f(t):=F(t_{1},\cdot\cdot \cdot, t_{j-1},-t-|a_{j}|, t_{j+1},\cdot \cdot \cdot, t_{n}),$ $t\geq 0.$\vspace{0.1cm}

In the remainder of this subsection, we will assume that $|c_{j}|=1$ for all $j\in {\mathbb N}_{n}.$ Then any semi-$(c_{j},{\mathcal B})_{j\in {\mathbb N}_{n}}$-periodic function $F : I\times X \rightarrow Y$ is bounded on any subset $B$ of the collection ${\mathcal B},$ as easily approved; even in the one-dimensional setting, this function need not be 
periodic in the usual sense (see \cite[p. 2]{JM}). Furthermore, if for each integer $k\in {\mathbb N}$ the function $ F_{k} :  I \times X \rightarrow Y$ is semi-$(c_{j},{\mathcal B})_{j\in {\mathbb N}_{n}}$-periodic and
for each $B\in {\mathcal B}$ there exists a finite real number $\epsilon_{B}>0$ such that
$\lim_{k\rightarrow +\infty}F_{k}({\bf t};x)=F({\bf t};x)$ for all ${\bf t}\in I$, uniformly in $x\in B^{\cdot} \equiv B^{\circ} \cup \bigcup_{x\in \partial B}B(x,\epsilon_{B}),$ then the function $F(\cdot; \cdot)$ is likewise semi-$(c_{j},{\mathcal B})_{j\in {\mathbb N}_{n}}$-periodic. 

Let $B\in {\mathcal B}$ be fixed. In what follows, we consider the Banach space $l_{\infty}(B : Y)$ consisting of all bounded functions $f : B \rightarrow Y,$ equipped with the sup-norm. Suppose that the function $ F :  I \times X \rightarrow Y$ is semi-$(c_{j},{\mathcal B})_{j\in {\mathbb N}_{n}}$-periodic.
Define the function
$F_{B} : I \rightarrow l_{\infty}(B : Y)$ by 
\begin{align*}
\bigl[F_{B}({\bf t})\bigr](x):=F({\bf t}; x),\quad {\bf t}\in I,\ x\in B.
\end{align*}
Then the mapping $F_{B}(\cdot)$ is well defined and semi-$(c_{j})_{j\in {\mathbb N}_{n}}$-periodic.
Using now an insignificant modification of the proofs of \cite[Lemma 1, Theorem 1]{andres1}, we may conclude that
for each set $B\in {\mathcal B}$ there exists a sequence of $(c_{j})_{j\in {\mathbb N}_{n}}$-periodic functions $(F_{k} : I\times X \rightarrow Y)_{k\in {\mathbb N}}$
such that $\lim_{k\rightarrow +\infty}F_{k}({\bf t};x)=F({\bf t};x)$ for all ${\bf t}\in I$, uniformly in $x\in B.$ The converse
statement is also true; hence, we have
the following important result:

\begin{thm}\label{reza-rezace}
Suppose that $F : I\times X \rightarrow Y$ is continuous. Then
the function $F(\cdot;\cdot)$ is semi-$(c_{j},{\mathcal B})_{j\in {\mathbb N}_{n}}$-periodic if and only if for each set $B\in {\mathcal B}$ there exists a sequence of $(c_{j})_{j\in {\mathbb N}_{n}}$-periodic functions $(F_{k} : I \rightarrow l_{\infty}(B : Y))_{k\in {\mathbb N}}$
such that $\lim_{k\rightarrow +\infty}F_{k}({\bf t})=F_{B}({\bf t})$ uniformly in ${\bf t}\in I.$
\end{thm}

Now we would like to present the following illustrative application of Theorem \ref{reza-rezace}:

\begin{example}\label{hieber}
Suppose that $q_{1},\cdot \cdot \cdot,q_{n}$ are odd natural numbers.
Define $F: {\mathbb R}^{n} \rightarrow {\mathbb C}$ by
$$
F\bigl(t_{1},t_{2},\cdot \cdot \cdot, t_{n} \bigr):=\sum_{l=(l_{1},l_{2},\cdot \cdot \cdot, l_{n})\in {\mathbb N}^{n}} \frac{e^{it_{1}/(2l_{1}q_{1}+1)} e^{it_{2}/(2l_{2}q_{2}+1)} \cdot \cdot \cdot e^{it_{n}/(2l_{n}q_{n}+1)} }{l_{1}! l_{2}! \cdot \cdot \cdot l_{n}!},
$$
for any ${\bf t}=(t_{1},t_{2},\cdot \cdot \cdot, t_{n}) \in {\mathbb R}^{n}.$ Then $F(\cdot)$ is semi-$(-1,-1,\cdot \cdot \cdot,-1)$-periodic function since it is a uniform limit of $(-1,-1,\cdot \cdot \cdot,-1)$-periodic functions
$$
F_{k}({\bf t}):=\sum_{|l|\leq k} \frac{e^{it_{1}/(2l_{1}q_{1}+1)} e^{it_{2}/(2l_{2}q_{2}+1)} \cdot \cdot \cdot e^{it_{n}/(2l_{n}q_{n}+1)} }{l_{1}! l_{2}! \cdot \cdot \cdot l_{n}!},\quad {\bf t}\in {\mathbb R}^{n},\ k\in {\mathbb N}.
$$
\end{example}

We continue with the observation that the statements of Proposition 2.5, Proposition 2.7, Proposition 2.8, Proposition 2.9, Proposition 2.12, Theorem 2.13 and Proposition 2.17 of \cite{c1} admit very simple reformulations in the multi-dimensional setting. For example, if $F : I \rightarrow {\mathbb R}$ is semi-$(c_{j} )_{j\in {\mathbb N}_{n}}$-periodic, then $c_{j}\in \{-1,1\}$ for all $j\in {\mathbb N}_{n};$ furthermore, if $F({\bf t})\geq 0$ for all ${\bf t}\in I,$ then $c_{j}=1$ for all $j\in {\mathbb N}_{n}.$ 

Any semi-$(c_{j})_{j\in {\mathbb N}_{n}}$-periodic function $F : I \rightarrow Y$  
can be extended uniquely to a semi-$(c_{j})_{j\in {\mathbb N}_{n}}$-periodic function $\tilde{F} : {\mathbb R}^{n} \rightarrow Y$ and therefore it has a mean value as an almost periodic function;
see e.g., the proof of \cite[Theorem 2.36]{marko-manuel-ap}. Furthermore, any semi-$(c_{j})_{j\in {\mathbb N}_{n}}$-periodic function $F : I \rightarrow Y$ is semi-periodic. In the one-dimensional case, \cite[Lemma 2]{andres1} tells us that there exists a positive real number $\theta >0$ such that $\sigma(F)\subseteq \theta \cdot {\mathbb Q},$ which enables one to construct a great deal of almost periodic functions which are not semi-periodic. A similar situation holds in the multi-dimensional setting, when we have the following:

\begin{prop}\label{ref-cmb}
Suppose that the function $F : I \rightarrow Y$ is semi-$(c_{j})_{j\in {\mathbb N}_{n}}$-periodic, $\lambda=(\lambda_{1},\lambda_{2},\cdot \cdot \cdot,\lambda_{n}) \in \sigma(F)$ and
$\mu=(\mu_{1},\mu_{2},\cdot \cdot \cdot,\mu_{n}) \in \sigma(F).$
Then there exist non-zero real numbers $\omega_{j} \in {\mathbb R} \setminus \{0\}$ ($1\leq j\leq n$) such that 
$\lambda_{j}\omega_{j}\in 2\pi {\mathbb Z}$ and $\mu_{j}\omega_{j}\in 2\pi {\mathbb Z}$  for all $j\in {\mathbb N}_{n}.$
\end{prop}

\begin{proof}
By the foregoing, we may assume that $I={\mathbb R}^{n},$ $\lambda=\mu$ and 
$c_{j}=1$ for all $j\in {\mathbb N}_{n}.$ We will follow the proof of \cite[Corollary 4.5.4(d)]{a43} with appropriate modifications. First of all, note that $\lim_{k\rightarrow +\infty}k^{-1}\sum_{j=0}^{k-1}z^{j}=0$ if $|z|=1$ and $z\neq 1,$ while $\lim_{k\rightarrow +\infty}k^{-1}\sum_{j=0}^{k-1}z^{j}=1$ if $z=1.$ Our assumption is that
$$
\lim_{T\rightarrow +\infty}\frac{1}{T^{n}}\int_{[0,T]^{n}}e^{-i \langle \lambda , {\bf t} \rangle}F({\bf t})\, d{\bf t}\neq 0\mbox{ and }\lim_{T\rightarrow +\infty}\frac{1}{T^{n}}\int_{[0,T]^{n}}e^{-i \langle \mu , {\bf t} \rangle}F({\bf t})\, d{\bf t}\neq 0.
$$
By Theorem \ref{reza-rezace}, the proof of \cite[Lemma 2]{andres1} and continuity, we may assume without loss of generality that $F(\cdot)$ is $(\omega_{j},c_{j})_{j\in {\mathbb N}_{n}}$-periodic for some non-zero real numbers $\omega_{j}\in {\mathbb R}\setminus \{0\}$ ($1\leq j\leq n$). We have{\small
\begin{align*}
&\lim_{T\rightarrow +\infty}\frac{1}{T^{n}}\int_{[0,T]^{n}}e^{-i \langle \lambda , {\bf t} \rangle}F({\bf t})\, d{\bf t}
\\&=\lim_{T\rightarrow +\infty}\frac{1}{T^{n}}\sum_{j_{1}=0}^{\lfloor T/|\omega_{1}| \rfloor}
\cdot \cdot \cdot \sum_{j_{n}=0}^{\lfloor T/|\omega_{n}|\rfloor}\int_{\prod_{k=1}^{n}[j_{k}|\omega_{k}|,(j_{k}+1)|\omega_{k}|]}e^{-i \langle \lambda , {\bf t} \rangle}F({\bf t})\, d{\bf t}
\\& =\lim_{T\rightarrow +\infty}\frac{1}{T^{n}}\sum_{j_{1}=0}^{\lfloor T/|\omega_{1}| \rfloor}
\cdot \cdot \cdot \sum_{j_{n}=0}^{\lfloor T/|\omega_{n}|\rfloor}\int_{[0,|\omega_{1}|] \times \cdot \cdot \cdot \times [0,|\omega_{n}|]}e^{i \bigl[\lambda_{1}j_{1}|\omega_{1}|+\cdot \cdot \cdot +\lambda_{n}j_{n}|\omega_{n}|\bigr]}
e^{-i \langle \lambda , {\bf t} \rangle}F({\bf t})\, d{\bf t}
\\& =\lim_{T\rightarrow +\infty}\Biggl\{\Biggl[\frac{1}{T}\sum_{j_{1}=0}^{\lfloor T/|\omega_{1}| \rfloor} \Bigl( e^{i\lambda_{1}|\omega_{1}|}\Bigr)^{j_{1}}\Biggr] \cdot \cdot \cdot 
\Biggl[\frac{1}{T}\sum_{j_{n}=0}^{\lfloor T/|\omega_{n}| \rfloor} \Bigl( e^{i\lambda_{n}|\omega_{n}|}\Bigr)^{j_{n}}\Biggr] \Biggr\}
\\& =\lim_{T\rightarrow +\infty}\Biggl[\frac{1}{T}\sum_{j_{1}=0}^{\lfloor T/|\omega_{1}| \rfloor} \Bigl( e^{i\lambda_{1}|\omega_{1}|}\Bigr)^{j_{1}}\Biggr] \cdot \cdot \cdot \lim_{T\rightarrow +\infty}\Biggl[\frac{1}{T}\sum_{j_{n}=0}^{\lfloor T/|\omega_{n}| \rfloor} \Bigl( e^{i\lambda_{n}|\omega_{n}|}\Bigr)^{j_{n}}\Biggr]. 
\end{align*}}
The final conclusion follows by observing that the product of above limits, which exist in ${\mathbb C},$ is not equal to zero if and only if $\exp(i\lambda_{j}|\omega_{j}|)=1$ for all $j\in {\mathbb N}_{n},$ as well as that the same calculation can be given for the tuple $\mu.$
\end{proof}

The Stepanov classes of semi-$(c_{j},{\mathcal B})_{j\in {\mathbb N}_{n}}$-periodic functions can be also analyzed; see \cite[Section 3]{Chaouchi} for more details given in the one-dimensional setting.

\section{Multi-dimensional quasi-asymptotically $c$-almost periodic type functions}\label{profica-efg}

In this section, we investigate several various classes of multi-dimensional quasi-asymptotically $c$-almost periodic functions. We
start by introducing the notion of ${\mathbb D}$-quasi-asymptotical $({\mathcal B},I',c)$-almost periodicity and recall the notion of 
${\mathbb D}$-quasi-asymptotical $({\mathcal B},I',c)$-uniform recurrence here (it can be easily shown that the notion of quasi-asymptotical uniform recurrence introduced in \cite[Definition 9]{dumath2}, with ${\mathbb D}=I=I'={\mathbb R},$ is equivalent with the corresponding notion introduced in the second part of the following definition; concerning the first part of this definition, it extends the notion of quasi-asymptotical $c$-almost
periodicity introduced in \cite[Definition 3.3]{BIMV}):

\begin{defn}\label{nakaza}
Suppose that 
${\mathbb D} \subseteq I \subseteq {\mathbb R}^{n}$, 
$\emptyset  \neq I'\subseteq I \subseteq {\mathbb R}^{n},$ 
the sets ${\mathbb D}$ and $I'$ are unbounded,
$F : I \times X \rightarrow Y$ is a continuous function and $I +I' \subseteq I.$ Then we say that:
\begin{itemize}
\item[(i)]\index{function!${\mathbb D}$-quasi-asymptotically Bohr $({\mathcal B},I',c)$-almost periodic}
$F(\cdot;\cdot)$ is ${\mathbb D}$-quasi-asymptotically $({\mathcal B},I',c)$-almost periodic if and only if for every $B\in {\mathcal B}$ and $\epsilon>0$
there exists $l>0$ such that for each ${\bf t}_{0} \in I'$ there exists ${\bf \tau} \in B({\bf t}_{0},l) \cap I'$ such that there exists a finite real number $M(\epsilon,\tau)>0$ such that
\begin{align}\label{emojmarko145m}
\bigl\|F({\bf t}+{\bf \tau};x)-cF({\bf t};x)\bigr\|_{Y} \leq \epsilon,\mbox{ provided } {\bf t},\ {\bf t}+\tau \in {\mathbb D}_{M(\epsilon,\tau)},\ x\in B.
\end{align}
\item[(ii)] \index{function!${\mathbb D}$-quasi-asymptotically $({\mathcal B},I',c)$-uniformly recurrent}
$F(\cdot;\cdot)$ is ${\mathbb D}$-quasi-asymptotically $({\mathcal B},I',c)$-uniformly recurrent if and only if for every $B\in {\mathcal B}$ 
there exist a sequence $({\bf \tau}_{k})$ in $I'$ and a sequence $(M_{k})$ in $(0,\infty)$ such that $\lim_{k\rightarrow +\infty} |{\bf \tau}_{k}|=\lim_{k\rightarrow +\infty}M_{k}=+\infty$ and
$$
\lim_{k\rightarrow +\infty}\sup_{{\bf t},{\bf t}+{\bf \tau}_{k}\in {\mathbb D}_{M_{k}};x\in B} \bigl\|F({\bf t}+{\bf \tau}_{k};x)-c F({\bf t};x)\bigr\|_{Y} =0.
$$
\end{itemize}
If $I'=I,$ then we also say that
$F(\cdot;\cdot)$ is ${\mathbb D}$-quasi-asymptotically $({\mathcal B},c)$-almost periodic (${\mathbb D}$-quasi-asymptotically $({\mathcal B},c)$-uniformly recurrent); furthermore, if $X\in {\mathcal B},$ then it is also said that $F(\cdot;\cdot)$ is ${\mathbb D}$-quasi-asymptotically $(I',c)$-almost periodic (${\mathbb D}$-quasi-asymptotically $(I',c)$-uniformly recurrent). If $I'=I$ and $X\in {\mathcal B}$, then we also say that $F(\cdot;\cdot)$ is ${\mathbb D}$-quasi-asymptotically  $c$-almost periodic (${\mathbb D}$-quasi-asymptotically $c$-uniformly recurrent). We remove the prefix ``${\mathbb D}$-'' in the case that ${\mathbb D}=I$, remove the prefix ``$({\mathcal B},)$''  in the case that $X\in {\mathcal B}$ and remove the prefix ``$c$-'' if $c=1.$ 
\end{defn}

In \cite{marko-manuel-ap} and \cite{multi-ce}, we have also analyzed the notion of ${\mathbb D}$-asymptotical Bohr $({\mathcal B},I',c)$-almost periodicity of type $1$, which is a special case of ${\mathbb D}$-quasi-asymptotical $({\mathcal B},I',c)$-almost periodicity.
The notion of 
${\mathbb D}$-quasi-asymptotical $({\mathcal B},I',c)$-uniform recurrence, which generalizes the notion of ${\mathbb D}$-quasi-asymptotical $({\mathcal B},I',c)$-almost periodicity, has been also introduced in \cite[Definition 2.25(ii)]{multi-ce}, under the slightly different name of 
${\mathbb D}$-asymptotical $({\mathcal B},I',c)$-uniform recurrence of type $1.$ 
It is evident that the notion of ${\mathbb D}$-asymptotical Bohr $({\mathcal B},I',c)$-almost periodicity of type $1$ 
(see Definition \ref{nafaks123456789012345123}) is  a special case of the notion of ${\mathbb D}$-quasi-asymptotical $({\mathcal B},I',c)$-almost periodicity introduced in Definition \ref{nakaza}(i).
The following generalization of \cite[Proposition 2]{dumath2} can be deduced straightforwardly (we can simply formulate an extension of \cite[Proposition 3]{dumath2} in the multi-dimensional setting, as well):

\begin{prop}\label{totijerez}
Suppose that 
${\mathbb D} \subseteq I \subseteq {\mathbb R}^{n},$ $c\in {\mathbb C} \setminus \{0\}$ and the set ${\mathbb D}$ is unbounded, as well as
$\emptyset  \neq I'\subseteq I \subseteq {\mathbb R}^{n},$ $F : I \times X \rightarrow Y$ is a continuous function and $I +I' \subseteq I.$ If the function 
$F(\cdot;\cdot)$ is ${\mathbb D}$-quasi-asymptotically $({\mathcal B},I',c)$-almost periodic, resp. ${\mathbb D}$-quasi-asymptotically $({\mathcal B},I',c)$-uniformly recurrent, and $Q\in C_{0,{\mathbb D},{\mathcal B}}(I\times X : Y),$ then
$[F+Q](\cdot;\cdot)$ is ${\mathbb D}$-quasi-asymptotically $({\mathcal B},I',c)$-almost periodic, resp. ${\mathbb D}$-quasi-asymptotically $({\mathcal B},I',c)$-uniformly recurrent.
\end{prop}

We continue by providing an illustrative example:

\begin{example}\label{notaur}
The function $F : {\mathbb R}^{n} \rightarrow {\mathbb R},$ given by $F({\bf t}):=\sin(\ln(1+|{\bf t}|)),$ ${\bf t}\in {\mathbb R}^{n},$ is quasi-asymptotically almost periodic but
not asymptotically uniformly recurrent; this can be shown as in the one-dimensional case (see \cite[Example 3]{dumath2}). Furthermore, it can be easily shown that $F(\cdot)$ is quasi-asymptotically $c$-almost periodic for some $c\in {\mathbb C} \setminus \{0\}$ if and only if $c=1.$
\end{example}

In the following result, we show that the notion introduced in the previous section can be viewed as a particular case of the notion introduced in Definition \ref{nakaza}(i), under some very reasonable assumptions (in the second part, we can also consider the situation in which
$I':=\omega_{j}e_{j} \cdot {\mathbb N}$ for some $j\in {\mathbb N}_{n}$):

\begin{prop}\label{gasdas}
\begin{itemize}
\item[(i)] Let ${\bf \omega}\in I \setminus \{0\},$ $c\in {\mathbb C} \setminus \{0\},$ $|c|\leq 1,$
${\bf \omega}+I \subseteq I$ and ${\mathbb D} \subseteq I \subseteq {\mathbb R}^{n}.$ Set $I':=\omega \cdot {\mathbb N}.$
If a continuous
function $F:I \times X\rightarrow Y$ is $(S,{\mathbb D},{\mathcal B})$-asymptotically 
$(\omega,c)$-periodic, then the function $F(\cdot; \cdot)$ is ${\mathbb D}$-quasi-asymptotically $({\mathcal B},I',c)$-almost periodic.
\item[(ii)] Let ${\bf \omega}_{j}\in {\mathbb R} \setminus \{0\},$ $c_{j}\in {\mathbb C} \setminus \{0\},$
${\bf \omega}_{j}e_{j}+I \subseteq I,$ 
${\mathbb D}_{j} \subseteq I \subseteq {\mathbb R}^{n},$ the set ${\mathbb D}_{j}$ be unbounded ($1\leq j\leq n$) and the set ${\mathbb D}$ consisting of all tuples ${\bf t}\in {\mathbb D}_{n}$ such that ${\bf t}+\sum_{i=j+1}^{n}\omega_{i}e_{i}$ for all $j\in {\mathbb N}_{n-1}$
be unbounded in ${\mathbb R}^{n}$. Set $\omega:=\sum_{j=1}^{n}\omega_{j}e_{j},$ 
$I':=\omega \cdot {\mathbb N}$
and $c:=\prod_{j=1}^{n}c_{j}.$
If $F:I \times X \rightarrow Y$ is $(S,{\mathcal B})$-asymptotically $({\bf \omega}_{j},c_{j},{\mathbb D}_{j})_{j\in {\mathbb N}_{n}}$-periodic, $|c|\leq 1$ and $\omega \in I,$ then  
the function $F(\cdot;\cdot)$ is ${\mathbb D}$-quasi-asymptotically $({\mathcal B},I',c)$-almost periodic.
\end{itemize}
\end{prop}

\begin{proof}
The proof of (i) is very similar to the proof of \cite[Proposition 3.2]{BIMV}. First of all, note that our assumptions ${\bf \omega}\in I \setminus \{0\}$ and ${\mathbb D} +\omega \cdot {\mathbb N}_{0} \subseteq {\mathbb D}$ imply that the set ${\mathbb D}$ is unbounded, whilst the assumptions ${\bf \omega}\in I \setminus \{0\}$ and $\omega +I \subseteq I$
imply that $I'$ is an unbounded subset of $I$ and 
$I +I' \subseteq I.$
Let $B\in {\mathcal B}$ and $\epsilon>0$ be fixed. Then we can take $l=2|\omega|$ in Definition \ref{nakaza}(i) since for each ${\bf t}_{0}=n'\omega \in I',$ where $n'\in {\mathbb N},$ there exists ${\bf \tau}=n\omega \in B({\bf t}_{0},l) \cap I'$, with $n'=n+1.$ Since the function $F(\cdot;\cdot)$ is  $(S,{\mathbb D},{\mathcal B})$-asymptotically 
$(\omega,c)$-periodic,
we have the existence of a finite real number $M>0$ such that the assumptions $|{\bf t}|>M$ and $
{\bf t}\in {\mathbb D}$ imply $\| F({\bf t}+\omega ; x)-cF({\bf t} ; x)\|<\epsilon /n$ for all $x\in B.$ 
Let ${\bf t}\in {\mathbb D}$ and $|{\bf t}|>M(\epsilon,\tau)\equiv  M+n|\omega|.$
Then
\eqref{emojmarko145m} holds since the assumptions ${\bf t},\ {\bf t}+\tau \in {\mathbb D}_{M(\epsilon,\tau)}$ and $ x\in B$ imply:
\begin{align*}
\bigl\| & F({\bf t}+{\bf \tau};x)-cF({\bf t};x)\bigr\|_{Y} 
\\& \leq \sum_{k=0}^{n-1}|c|^{n-k-1}\Bigl\| F({\bf t}+(k+1)\omega;x)-cF({\bf t}+k\omega;x) \Bigr\|_{Y}
\\& \leq \sum_{k=0}^{n-1} \Bigl\| F({\bf t}+(k+1)\omega;x)-cF({\bf t}+k\omega;x) \Bigr\|_{Y} 
\leq n(\epsilon/n)=\epsilon,
\end{align*}
as claimed.
To deduce (ii), it suffices to observe that our assumptions imply by Proposition \ref{drale}
that the function $F(\cdot;\cdot)$ is $(S,{\mathbb D},{\mathcal B})$-asymptotically 
$(\omega,c)$-periodic, with $\omega=\sum_{j=1}^{n}\omega_{j}e_{j}.$
After that, we can apply the first part of proposition.
\end{proof}

The spaces introduced in Definition \ref{nakaza} do not form vector spaces under the pointwise addition of functions and these spaces are not closed under the pointwise multiplication with scalar-valued functions of the same type, as is well known in the one-dimensional case (\cite{brazil}).
The introduced spaces are homogeneous and, under certain reasonable assumptions, these spaces are translation invariant, invariant under the
homotheties with ratio $b>0$ and the reflections at zero with respect to the first variable; details can be left to the interested readers. 
Furthermore, we have the following statements stated here without simple proofs (see also \cite[Proposition 2.7, Proposition 2.8]{marko-manuel-ap}):

\begin{prop}\label{jkljkl}
\begin{itemize}
\item[(i)] Suppose that 
${\mathbb D} \subseteq I \subseteq {\mathbb R}^{n}$, 
$\emptyset  \neq I'\subseteq I \subseteq {\mathbb R}^{n},$ 
the sets ${\mathbb D}$ and $I'$ are unbounded,
$F : I \times X \rightarrow {\mathbb C}$ is a continuous function and $I +I' \subseteq I.$ 
\begin{itemize}
\item[(i)] If
$F(\cdot;\cdot)$ is ${\mathbb D}$-quasi-asymptotically $({\mathcal B},I',c)$-almost periodic and, for every $B\in {\mathcal B},$ 
there exists a real number $c_{B}>0$ such that $| F({\bf t} ; x)|\geq c_{B}$ for all $x\in B$ and ${\bf t}\in I,$
then the function $1/F(\cdot ; \cdot)$ is ${\mathbb D}$-quasi-asymptotically $({\mathcal B},I',1/c)$-almost periodic.
\item[(ii)] \index{function!${\mathbb D}$-quasi-asymptotically $({\mathcal B},I',c)$-uniformly recurrent  of type $1$}
$F(\cdot;\cdot)$ is ${\mathbb D}$-quasi-asymptotically $({\mathcal B},I',c)$-uniformly recurrent if and only if for every $B\in {\mathcal B}$ 
there exist a sequence $({\bf \tau}_{k})$ in $I'$ and a sequence $(M_{k})$ in $(0,\infty)$ such that $\lim_{k\rightarrow +\infty} |{\bf \tau}_{k}|=\lim_{k\rightarrow +\infty}M_{k}=+\infty$ and
$$
\lim_{k\rightarrow +\infty}\sup_{{\bf t},{\bf t}+{\bf \tau}_{k}\in {\mathbb D}_{M_{k}};x\in B} \bigl\|F({\bf t}+{\bf \tau}_{k};x)-c F({\bf t};x)\bigr\|_{Y} =0.
$$
\end{itemize}
\item[(ii)] If $(F_{k}(\cdot;\cdot))$ is a sequence of ${\mathbb D}$-quasi-asymptotically $({\mathcal B},I',c)$-almost periodic functions, resp. ${\mathbb D}$-quasi-asymptotically $({\mathcal B},I',c)$-uniformly recurrent functions, such that 
for each $B\in {\mathcal B}$ there exists a finite real number $\epsilon_{B}>0$ such that
$\lim_{k\rightarrow +\infty}F_{k}({\bf t};x)=F({\bf t};x)$ for all ${\bf t}\in {\mathbb R}$, uniformly in $x\in B^{\cdot} \equiv B^{\circ} \cup \bigcup_{x\in \partial B}B(x,\epsilon_{B}),$ then the function $F(\cdot; \cdot)$ is 
 ${\mathbb D}$-quasi-asymptotically $({\mathcal B},I',c)$-almost periodic, resp. ${\mathbb D}$-quasi-asymptotically $({\mathcal B},I',c)$-uniformly recurrent.
\end{itemize}
\end{prop}

The proof of following result is very similar to that of Theorem \ref{milenko} and therefore omitted (the assumption on compact support of function $h(\cdot)$ made in \cite{dumath2} for the class of quasi-asymptotically uniformly recurrent functions is superfluous):

\begin{thm}\label{milenko1}
Suppose that $h\in L^{1}({\mathbb R}^{n}),$ $\emptyset  \neq I'\subseteq {\mathbb R}^{n}$ is unbounded and $F : {\mathbb R}^{n} \times X \rightarrow Y$ is a continuous function satisfying that for each $B\in {\mathcal B}$ there exists a finite real number $\epsilon_{B}>0$ such that
$\sup_{{\bf t}\in {\mathbb R}^{n},x\in B^{\cdot}}\|F({\bf t},x)\|_{Y}<+\infty,$
where $B^{\cdot} \equiv B^{\circ} \cup \bigcup_{x\in \partial B}B(x,\epsilon_{B}).$ 
Suppose that ${\mathbb D}={\mathbb R}^{n}.$ Then the function $(h\ast F)(\cdot;\cdot),$ given by 
\eqref{gariprekrsaj},
is well defined and for each $B\in {\mathcal B}$ we have $\sup_{{\bf t}\in {\mathbb R}^{n},x\in B^{\cdot}}\|(h\ast F)({\bf t};x)\|_{Y}<+\infty;$ furthermore, if $F(\cdot;\cdot)$ is 
${\mathbb R}^{n}$-quasi-asymptotically $({\mathcal B},I',c)$-almost periodic, resp. ${\mathbb R}^{n}$-quasi-asymptotically $({\mathcal B},I',c)$-uniformly recurrent,
then the function $(h\ast F)(\cdot;\cdot)$ is likewise ${\mathbb R}^{n}$-quasi-asymptotically $({\mathcal B},I',c)$-almost periodic,
resp.\\ ${\mathbb R}^{n}$-quasi-asymptotically $({\mathcal B},I',c)$-uniformly recurrent.
\end{thm}

Accepting the notation employed in \cite{brazil} and \cite{dumath2}, we have the following ($I={\mathbb R}$ or $I=[0,\infty);$ $\omega \in I$): 
\begin{itemize}
\item[(i)] Suppose that $f\in SAP_{\omega}({\mathbb R}: X) \cap AAA({\mathbb R} : X),$ resp. $f\in SAP_{\omega}(I: X) \cap AAP(I : X).$  Then $f\in AP_{\omega}({\mathbb R}:X),$ resp. $f\in AP_{\omega}(I:X).$
\item[(ii)] Suppose that $f\in SAP_{\omega}({\mathbb R}: X) \cap AA({\mathbb R} : X),$
resp.
$f\in SAP_{\omega}(I: X) \cap AP(I : X).$ Then $f\in C_{\omega}({\mathbb R}:X),$ resp. $f\in C_{\omega}(I:X).$
\item[(iii)] $AAA({\mathbb R} : X) \cap Q-AAP({\mathbb R} : X)=AAP({\mathbb R}: X)$ and $[AAA({\mathbb R} : X) \setminus AAP({\mathbb R} : X)] \cap Q-AAP({\mathbb R}:X)=\emptyset.$
\item[(iv)] $AA({\mathbb R}: X) \cap Q-AAP({\mathbb R}: X)=AP({\mathbb R}: X).$
\item[(v)]
Let ${\mathrm F}(I : X)$ be any space of functions $h : I \rightarrow X$ satisfying that for each $\tau \in I$ the supremum formula holds for
the function $h(\cdot+\tau)-h(\cdot),$ i.e.,
$$
\sup_{t\in I}\| h(\cdot+\tau)-h(\cdot)\|=\sup_{t\in I,\ |t|\geq a}\| h(\cdot+\tau)-h(\cdot)\|,\quad a\in I.
$$
Then we have:
$[{\mathrm F}(I : X) +C_{0}(I:X) ]\cap Q-AUR(I : X) \subseteq AUR(I: X)$
and ${\mathrm F}(I : X) \cap Q-AUR(I : X) \subseteq UR(I: X).$
\end{itemize}
Furthermore, the above statements can be reformulated for the corresponding Stepanov classes. 

We will only note here that these statements admit very simple generalizations in the multi-dimensional setting. For example, if $I={\mathbb R}^{n}$ or $I=[0,\infty)^{n}$ and
the function $F : I \rightarrow Y$ is both
$S$-asymptotically $(\omega_{j},c_{j},I)_{j\in {\mathbb N}_{n}}$-periodic and $I$-asymptotically Bohr $(I,1)$-almost periodic, then the function $F(\cdot)$ is $(\omega_{j},c_{j})_{j\in {\mathbb N}_{n}}$-periodic (see also \cite[Example 4]{dumath2}, which can be used to provide certain examples of compactly almost automorphic functions in ${\mathbb R}^{n}$ which are not quasi-asymptotically uniformly recurrent \cite{marko-manuel-aa}-\cite{stmarko-manuel-aa}). The crucial thing is that the proof of \cite[Theorem 1]{dumath2}
works in the multi-dimensional setting (see the item (v) above).

\section{Stepanov classes of quasi-asymptotically $c$-almost periodic type functions}\label{zlatnidecko}

In this section, we investigate the Stepanov classes of quasi-asymptotically $c$-almost periodic type functions (the Weyl and Besicovitch generalizations of quasi-asymptotically $c$-almost periodic type functions can be also introduced and analyzed but we will skip all related details concerning this issue here). We will always assume that $c\in {\mathbb C}\setminus \{0\},$
$\Omega$ is a fixed compact subset of ${\mathbb R}^{n}$ with positive Lebesgue measure, $\emptyset \neq \Lambda \subseteq {\mathbb R}^{n}$ satisfies $\Lambda +\Omega \subseteq \Lambda ,$ 
${\mathbb D} \subseteq \Lambda \subseteq {\mathbb R}^{n},$ $\emptyset  \neq  \Lambda'\subseteq  \Lambda \subseteq {\mathbb R}^{n}$,  the sets ${\mathbb D}$ and $\Lambda'$ are unbounded, as well as $ \Lambda + \Lambda' \subseteq \Lambda.$

We employ the following conditions:
\begin{enumerate}
\item[$(MD-B)_{S}:$] $\phi :[0,\infty) \rightarrow [0,\infty) ,$ $ p\in {\mathcal P}(\Omega),$  $\textsf{F} : \Lambda
    \times (0,\infty) \times \Lambda' \rightarrow (0,\infty),$
${\bf F} : \Lambda \times {\mathbb N} \rightarrow (0,\infty)$ and $\texttt{F} : \Lambda \rightarrow (0,\infty).$ 
\end{enumerate}

We will follow the approach obeyed for introduction of notion in \cite[Definition 13-Definition 15]{dumath2}, only, in which we do not loose the valuable  information about the translation invariance of introduced spaces (cf. also \cite[Definition 10-Definition 12]{dumath2}, where this is not the case):

\begin{defn}\label{gorilazurappf}
Let $(MD-B)_{S}$ hold.
\begin{enumerate}
\item[(i)] A function $F : \Lambda \times X \rightarrow Y$ is called
Stepanov-$[\Omega, {\mathcal B}, \Lambda',{\mathbb D}, p,\phi,\textsf{F},c]$-quasi-asymptotically almost periodic, resp. Stepanov-$[\Omega, {\mathcal B}, \Lambda',{\mathbb D}, p,\phi,{\bf F},c]$-quasi-asymptotically uniformly recurrent,
if and only if 
for every $B\in {\mathcal B}$ and $\epsilon>0$
there exists $l>0$ such that for each ${\bf t}_{0} \in \Lambda'$ there exists ${\bf \tau} \in B({\bf t}_{0},l) \cap \Lambda'$ such that there exists a finite real number $M(\epsilon,\tau)>0$ such that
\begin{align}\label{zeleni}
\sup_{{\bf t}\in {\mathbb D}_{M(\epsilon,\tau)}: {\bf t}+\tau \in {\mathbb D}_{M(\epsilon,\tau)}; x\in B}\textsf{F}({\bf t},\epsilon,\tau)\phi\bigl(\bigl\| F(\cdot+{\bf t}+\tau ;x)-c F(\cdot +{\bf t}; x)\bigr\|_{Y}\bigr)_{L^{p(\cdot)}(\Omega)}\leq
\epsilon,
\end{align}
resp.
there exist a strictly increasing sequence $(\tau_{k})$ in $\Lambda'$ whose norms tending to plus infinity and a sequence $(M_{k})$
of positive real numbers tending to plus infinity such that
\begin{align*}
\lim_{k\rightarrow +\infty}\sup_{{\bf t}\in {\mathbb D}_{M_{k}} : {\bf t}+\tau_{k}\in  {\mathbb D}_{M_{k}}; x\in B}{\bf F}({\bf t},k)\phi\bigl(\bigl\| F(\cdot+{\bf t}+\tau_{k};x)-cF(\cdot +{\bf t};x)\bigr\|_{Y}\bigr)_{L^{p(\cdot)}(\Omega)}=0.
\end{align*}
\item[(ii)]
Let ${\bf \omega}\in {\mathbb R}^{n} \setminus \{0\},$ $c\in {\mathbb C} \setminus \{0\},$
${\bf \omega}+\Lambda\subseteq \Lambda,$ ${\mathbb D} \subseteq \Lambda \subseteq {\mathbb R}^{n}$ and the set ${\mathbb D}$ be unbounded. A 
function $F: \Lambda \times X\rightarrow Y$ is said to be Stepanov $[S,\Omega, {\mathcal B}, {\mathbb D}, p,\phi,\texttt{F}]$-asymptotically
$(\omega,c)$-periodic if and only if for each $B\in {\mathcal B}$ we have{\small
\begin{align*}
\lim_{|{\bf t}|\rightarrow +\infty,{\bf t}\in {\mathbb D}}\texttt{F}({\bf t}) \phi\bigl(\bigl\| F({\bf t}+\omega +\cdot ; x)-cF({\bf t}+\cdot ; x)\bigr\|_{Y}\bigr)_{L^{p(\cdot)}(\Omega)}=0,\ \mbox{ uniformly in }x\in B.
\end{align*}}
\item[(iii)]
Let ${\bf \omega}_{j}\in {\mathbb R} \setminus \{0\},$ $c_{j}\in {\mathbb C} \setminus \{0\},$
${\bf \omega}_{j}e_{j}+\Lambda \subseteq \Lambda$,
${\mathbb D}_{j} \subseteq \Lambda \subseteq {\mathbb R}^{n}$ and the set ${\mathbb D}_{j}$ be unbounded ($1\leq j\leq n$).
A 
function $F:\Lambda \times X\rightarrow Y$ is said to be  $[S, \Omega, {\mathcal B}, {\mathbb D}, p,\phi,\texttt{F}]$-asymptotically $({\bf \omega}_{j},c_{j},{\mathbb D}_{j})_{j\in {\mathbb N}_{n}}$-periodic if and only if for each $ j\in {\mathbb N}_{n}$ we have{\small
\begin{align*}
\lim_{|{\bf t}|\rightarrow +\infty,{\bf t}\in {\mathbb D}_{j}}\texttt{F}({\bf t}) \phi\bigl(\bigl\|F({\bf t}+{\bf \omega}_{j}e_{j}+\cdot;x)-c_{j}F({\bf t}+\cdot;x)\bigr\|_{Y}\bigr)_{L^{p(\cdot)}(\Omega)}=0, \mbox{ uniformly in }x\in B.
\end{align*}}
\end{enumerate}
\end{defn}

\begin{defn}\label{gorilazurappfmd1}
Let $(MD-B)_{S}$ hold.
\begin{enumerate}
\item[(i)] A function $F : \Lambda \times X \rightarrow Y$ is called
Stepanov-$[\Omega, {\mathcal B}, \Lambda',{\mathbb D}, p,\phi,\textsf{F},c]$-quasi-asymptotically almost periodic of type $1$, resp.\\ Stepanov-$[\Omega, {\mathcal B}, \Lambda',{\mathbb D}, p,\phi,{\bf F},c]$-quasi-asymptotically uniformly recurrent  of type $1$,
if and only if 
for every $B\in {\mathcal B}$ and $\epsilon>0$
there exists $l>0$ such that for each ${\bf t}_{0} \in \Lambda'$ there exists ${\bf \tau} \in B({\bf t}_{0},l) \cap \Lambda'$ such that there exists a finite real number $M(\epsilon,\tau)>0$ such that
\begin{align*}
\sup_{{\bf t}\in {\mathbb D}_{M(\epsilon,\tau)}: {\bf t}+\tau \in {\mathbb D}_{M(\epsilon,\tau)}; x\in B}\textsf{F}({\bf t},\epsilon,\tau)\phi\Bigl(\bigl\| F(\cdot+{\bf t}+\tau ;x)-c F(\cdot +{\bf t}; x)\bigr\|_{L^{p(\cdot)}(\Omega :Y)}\Bigr)\leq
\epsilon,
\end{align*}
resp.
there exist a strictly increasing sequence $(\tau_{k})$ in $\Lambda'$ whose norms tending to plus infinity and a sequence $(M_{k})$
of positive real numbers tending to plus infinity such that
\begin{align*}
\lim_{k\rightarrow +\infty}\sup_{{\bf t}\in {\mathbb D}_{M_{k}} : {\bf t}+\tau_{k}\in  {\mathbb D}_{M_{k}}; x\in B}{\bf F}({\bf t},k)\phi\Bigl(\bigl\| F(\cdot+{\bf t}+\tau_{k};x)-cF(\cdot +{\bf t};x)\bigr\|_{L^{p(\cdot)}(\Omega :Y)}\Bigr)=0.
\end{align*}
\item[(ii)]
Let ${\bf \omega}\in {\mathbb R}^{n} \setminus \{0\},$ $c\in {\mathbb C} \setminus \{0\},$
${\bf \omega}+\Lambda\subseteq \Lambda,$ ${\mathbb D} \subseteq \Lambda \subseteq {\mathbb R}^{n}$ and the set ${\mathbb D}$ be unbounded. A 
function $F: \Lambda \times X\rightarrow Y$ is said to be Stepanov $[S,\Omega, {\mathcal B}, {\mathbb D}, p,\phi,\texttt{F}]$-asymptotically
$(\omega,c)$-periodic  of type $1$ if and only if for each $B\in {\mathcal B}$ we have{\small
\begin{align*}
\lim_{|{\bf t}|\rightarrow +\infty,{\bf t}\in {\mathbb D}}\texttt{F}({\bf t}) \phi\Bigl(\bigl\| F({\bf t}+\omega +\cdot ; x)-cF({\bf t}+\cdot ; x)\bigr\|_{L^{p(\cdot)}(\Omega :Y)}\Bigr)=0,\ \mbox{ uniformly in }x\in B.
\end{align*}}
\item[(iii)]
Let ${\bf \omega}_{j}\in {\mathbb R} \setminus \{0\},$ $c_{j}\in {\mathbb C} \setminus \{0\},$
${\bf \omega}_{j}e_{j}+\Lambda \subseteq \Lambda$,
${\mathbb D}_{j} \subseteq \Lambda \subseteq {\mathbb R}^{n}$ and the set ${\mathbb D}_{j}$ be unbounded ($1\leq j\leq n$).
A 
function $F:\Lambda \times X\rightarrow Y$ is said to be  $[S,\Omega, {\mathcal B}, {\mathbb D}, p,\phi,\texttt{F}]$-asymptotically $({\bf \omega}_{j},c_{j},{\mathbb D}_{j})_{j\in {\mathbb N}_{n}}$-periodic of type $1$ if and only if for each $ j\in {\mathbb N}_{n}$ we have{\small
\begin{align*}
\lim_{|{\bf t}|\rightarrow +\infty,{\bf t}\in {\mathbb D}_{j}}\texttt{F}({\bf t}) \phi\bigl(\bigl\|F({\bf t}+{\bf \omega}_{j}e_{j}+\cdot;x)-c_{j}F({\bf t}+\cdot;x)\bigr\|\bigr)_{L^{p(\cdot)}(\Omega :Y)}=0, \mbox{ uniformly in }x\in B.
\end{align*}}
\end{enumerate}
\end{defn}

\begin{defn}\label{gorilazurappfmd2}
Let $(MD-B)_{S}$ hold.
\begin{enumerate}
\item[(i)] A function $F : \Lambda \times X \rightarrow Y$ is called
Stepanov-$[\Omega, {\mathcal B}, \Lambda',{\mathbb D}, p,\phi,\textsf{F},c]$-quasi-asymptotically almost periodic of type $2$, resp.\\ Stepanov-$[\Omega, {\mathcal B}, \Lambda',{\mathbb D}, p,\phi,{\bf F},c]$-quasi-asymptotically uniformly recurrent of type $2$,
if and only if 
for every $B\in {\mathcal B}$ and $\epsilon>0$
there exists $l>0$ such that for each ${\bf t}_{0} \in \Lambda'$ there exists ${\bf \tau} \in B({\bf t}_{0},l) \cap \Lambda'$ such that there exists a finite real number $M(\epsilon,\tau)>0$ such that
\begin{align*}
\sup_{{\bf t}\in {\mathbb D}_{M(\epsilon,\tau)}: {\bf t}+\tau \in {\mathbb D}_{M(\epsilon,\tau)}; x\in B}\phi\Bigl( \textsf{F}({\bf t},\epsilon,\tau) \bigl\| F(\cdot+{\bf t}+\tau ;x)-c F(\cdot +{\bf t}; x)\bigr\|_{L^{p(\cdot)}(\Omega :Y)}\Bigr)\leq
\epsilon,
\end{align*}
resp.
there exist a strictly increasing sequence $(\tau_{k})$ in $\Lambda'$ whose norms tending to plus infinity and a sequence $(M_{k})$
of positive real numbers tending to plus infinity such that
\begin{align*}
\lim_{k\rightarrow +\infty}\sup_{{\bf t}\in {\mathbb D}_{M_{k}} : {\bf t}+\tau_{k}\in  {\mathbb D}_{M_{k}}; x\in B}\phi\Bigl( {\bf F}({\bf t},k) \bigl\| F(\cdot+{\bf t}+\tau_{k};x)-cF(\cdot +{\bf t};x)\bigr\|_{L^{p(\cdot)}(\Omega :Y)}\Bigr)=0.
\end{align*}
\item[(ii)]
Let ${\bf \omega}\in {\mathbb R}^{n} \setminus \{0\},$ $c\in {\mathbb C} \setminus \{0\},$
${\bf \omega}+\Lambda\subseteq \Lambda,$ ${\mathbb D} \subseteq \Lambda \subseteq {\mathbb R}^{n}$ and the set ${\mathbb D}$ be unbounded. A 
function $F: \Lambda \times X\rightarrow Y$ is said to be Stepanov $[S,\Omega, {\mathcal B}, {\mathbb D}, p,\phi,\texttt{F}]$-asymptotically
$(\omega,c)$-periodic of type $2$ if and only if for each $B\in {\mathcal B}$ we have{\small
\begin{align*}
\lim_{|{\bf t}|\rightarrow +\infty,{\bf t}\in {\mathbb D}} \phi\Bigl( \texttt{F}({\bf t})\bigl\| F({\bf t}+\omega +\cdot ; x)-cF({\bf t}+\cdot ; x)\bigr\|_{L^{p(\cdot)}(\Omega :Y)}\Bigr)=0,\ \mbox{ uniformly in }x\in B.
\end{align*}}
\item[(iii)]
Let ${\bf \omega}_{j}\in {\mathbb R} \setminus \{0\},$ $c_{j}\in {\mathbb C} \setminus \{0\},$
${\bf \omega}_{j}e_{j}+\Lambda \subseteq \Lambda$,
${\mathbb D}_{j} \subseteq \Lambda \subseteq {\mathbb R}^{n}$ and the set ${\mathbb D}_{j}$ be unbounded ($1\leq j\leq n$).
A 
function $F:\Lambda \times X\rightarrow Y$ is said to be  $[S,\Omega, {\mathcal B}, {\mathbb D}, p,\phi,\texttt{F}]$-asymptotically $({\bf \omega}_{j},c_{j},{\mathbb D}_{j})_{j\in {\mathbb N}_{n}}$-periodic of type $2$ if and only if for each $ j\in {\mathbb N}_{n}$ we have{\small
\begin{align*}
\lim_{|{\bf t}|\rightarrow +\infty,{\bf t}\in {\mathbb D}_{j}}\phi\Bigl( \texttt{F}({\bf t})  \bigl\|F({\bf t}+{\bf \omega}_{j}e_{j}+\cdot;x)-c_{j}F({\bf t}+\cdot;x)\bigr\|_{L^{p(\cdot)}(\Omega : Y)}\Bigr)=0, \mbox{ uniformly in }x\in B.
\end{align*}}
\end{enumerate}
\end{defn}

\begin{rem}\label{notion-de}
If ${\mathbb D}+\Lambda' \subseteq {\mathbb D}$ (this is always true provided that ${\mathbb D}=\Lambda$ due to our standing assumption), then it is irrelevant whether we will write $\sup_{{\bf t}\in {\mathbb D}_{M_{k}} : {\bf t}+\tau_{k}\in  {\mathbb D}_{M_{k}}}\cdot$ or only $\sup_{{\bf t}\in {\mathbb D}_{M_{k}}}\cdot$ in Definition \ref{gorilazurappf}(ii); a similar comment holds for the notion introduced in Definition \ref{gorilazurappf}(i), Definition \ref{gorilazurappfmd1} and Definition \ref{gorilazurappfmd2}.
\end{rem}

Without any doubt, the most intriguing case is that in which we have $p(x)\equiv p\in [1,\infty),$
$\phi(x)\equiv x,$ $\Omega =[0,1]^{n},$ 
and the functions $\textsf{F},\ F,\ \texttt{F}$ are identically equal to one. In this case, we can simply reformulate a great number of statements clarified by now for the Stepanov classes of functions introduced in this section by using the notion of multi-dimensional Bochner transform from \cite{stmarko-manuel-ap}. More precisely, for a given function $F : \Lambda \times X \rightarrow Y,$ we define its multi-dimensional Bochner transform $\hat{F}_{\Lambda} : \Lambda \times X \rightarrow Y^{\Omega}$ by 
$$
\Bigl[\hat{F}_{\Lambda}({\bf t};x)\Bigr]({\bf u}):=F({\bf t}+{\bf u};x),\quad {\bf t}\in \Lambda,\ {\bf u}\in \Omega,\ x\in B;
$$
here, $Y^{\Omega}$ denotes the collection of all functions $f : \Omega  \rightarrow Y.$ If $\hat{F}_{\Omega} : \Lambda \times X \rightarrow L^{p({\bf u})}(\Lambda : Y)$ is well defined and continuous, then the function $F: \Lambda \times X \rightarrow Y$ will be, e.g.,
Stepanov-$[\Omega, {\mathcal B}, \Lambda',{\mathbb D}, p,\phi,\textsf{F},c]$-quasi-asymptotically almost periodic
if and only if the function $\hat{F}_{\Omega} : \Lambda \times X \rightarrow L^{p({\bf u})}(\Lambda : Y)$ is ${\mathbb D}$-quasi-asymptotically $({\mathcal B},\Lambda',c)$-almost periodic. 
In the case that the functions $\textsf{F},\ F,\ \texttt{F}$ are only bounded and not necessarily identically equal to one, then we can simply transfer the statements of \cite[Proposition 4, Corollary 1]{dumath2} to the multi-dimensional setting. Details can be left to the interested readers.

Using the trivial inequalities and Lemma \ref{aux}, we can clarify a great number of inclusions for the introduced classes of functions (see \cite{dumath2} for more details given in the one-dimensional setting); 
for example, using Lemma \ref{aux}(iv) and a simple argumentation, we may deduce the following:

\begin{prop}\label{unijsereklamilica}
Let a function $F : \Lambda \times X \rightarrow Y$ be
Stepanov-$[\Omega, {\mathcal B}, \Lambda',{\mathbb D}, p,\phi,\textsf{F},c]$-quasi-asymptotically almost periodic, resp. Stepanov-$[\Omega, {\mathcal B}, \Lambda',{\mathbb D}, p,\phi,{\bf F},c]$-quasi-asymptotically uniformly recurrent, and let $A\in L(Y,Z).$ Then $AF : \Lambda \times X \rightarrow Z$ is likewise 
Stepanov-$[\Omega, {\mathcal B}, \Lambda',{\mathbb D}, p,\phi,\textsf{F},c]$-quasi-asymptotically almost periodic, resp. Stepanov-$[\Omega, {\mathcal B}, \Lambda',{\mathbb D}, p,\phi,{\bf F},c]$-quasi-asymptotically uniformly recurrent.
\end{prop}

The main result of this section, Theorem \ref{milenko2}, can be reworded for all other classes of functions introduced in Definition \ref{gorilazurappf}(ii)-(iii), Definition \ref{gorilazurappfmd1} and Definition \ref{gorilazurappfmd2}:

\begin{thm}\label{milenko2}
Let a function $F : {\mathbb R}^{n} \times X \rightarrow Y$ be
Stepanov-$[\Omega, {\mathcal B}, \Lambda',{\mathbb D}, p,\phi,\textsf{F},c]$-quasi-asymptotically almost periodic, resp. Stepanov-$[\Omega, {\mathcal B}, \Lambda',{\mathbb D}, p,\phi,{\bf F},c]$-quasi-asymptotically uniformly recurrent, where $\Omega=[0,1]^{n},$ ${\mathbb D}={\mathbb R}^{n},$
$\phi : [0,\infty) \rightarrow [0,\infty)$ is a convex, monotonically increasing function which additionally satisfies condition
\begin{itemize}
\item[(F)] There exists a function $\varphi : [0,\infty) \rightarrow [0,\infty)$ such that $\phi(xy)\leq \varphi(x)\phi(y) $ for all $x,\ y\geq 0.$
\end{itemize}
Let $h\in L^{1}({\mathbb R}^{n})$ and let for each set $B\in {\mathcal B}$ we have $\sup_{{\bf t}\in {\mathbb R}^{n};x\in B}\|F({\bf t};x)\|_{Y}<\infty.$
Suppose that there exists a continuous function $g : [0,\infty) \rightarrow [0,\infty)$ with $g(0+)=0+$ and a sequence $(a_{k})_{k\in {\mathbb Z}^{n}}$ of strictly increasing positive reals such that $\sum_{k\in {\mathbb Z}^{n}}a_{k}=1$ and
for each $\epsilon>0$ and $\tau \in \Lambda',$ resp. for each $n\in {\mathbb N}$ and $\tau \in \Lambda',$ there exists $M'(\epsilon,\tau)>0,$ resp. $M'(n,\tau)>0,$ such that for each ${\bf t} \in {\mathbb R}^{n}$ with $|{\bf t}|\geq M'(\epsilon,\tau),$ resp. $|{\bf t}|\geq M'(n,\tau),$ we have{\small
\begin{align}\label{mutacija}
\int_{[0,1]^{n}}\varphi_{p({\bf u})}\Biggl( \textsf{F}_{1}({\bf t}, \epsilon,\tau)\Biggl( \varphi(2)\sum_{k\in {\mathbb Z}^{n}}\frac{a_{k}\varphi(a_{k}^{-1})\bigl[ \varphi(|h({\bf t}-\sigma-k)|) \bigr]_{L^{q(\sigma)}(\Omega)}}{\textsf{F}({\bf u}+k,\epsilon,\tau)} +g(\epsilon)\Biggr)\Biggr)\, d{\bf u} \leq 1,
\end{align}}
resp.{\small 
$$
\int_{[0,1]^{n}}\varphi_{p({\bf u})}\Biggl( {\bf F}_{1}({\bf t}, n)\Biggl( \varphi(2)\sum_{k\in {\mathbb Z}^{n}}\frac{a_{k}\varphi(a_{k}^{-1})\bigl[ \varphi(|h({\bf t}-\sigma-k)|) \bigr]_{L^{q(\sigma)}(\Omega)}}{\textsf{F}({\bf u}+k,n)} +g(1/n)\Biggr)\Biggr)\, d{\bf u} \leq 1.
$$}
Then the function $(h\ast F)(\cdot;\cdot)$ is Stepanov-$[\Omega, {\mathcal B}, \Lambda',{\mathbb D}, p,\phi,\textsf{F}_{1},c]$-quasi-asymptotically almost periodic, resp. Stepanov-$[\Omega, {\mathcal B}, \Lambda',{\mathbb D}, p,\phi,{\bf F}_{1},c]$-quasi-asymptotically uniformly recurrent.
\end{thm}

\begin{proof}
We will prove the result only for the class of Stepanov-$[\Omega, {\mathcal B}, \Lambda',{\mathbb D}, p,\phi,\textsf{F},c]$-quasi-asymptotically almost periodic functions. It is clear that the function $(h\ast F)(\cdot;\cdot)$ is well defined. Let $\epsilon>0$ and $B\in {\mathcal B}$ be fixed. Due to our assumption, there exists $l>0$ s.t. for each ${\bf t}_{0} \in \Lambda'$ there exists ${\bf \tau} \in B({\bf t}_{0},l) \cap \Lambda'$ s.t. there exists a finite real number $M(\epsilon,\tau)>0$ s.t. \eqref{zeleni} holds. Let such a point $\tau$ be fixed. Then we know that there exists $M'(\epsilon,\tau)>0$ such that for each ${\bf t} \in {\mathbb R}^{n}$ with $|{\bf t}|\geq M'(\epsilon,\tau)$  we have
\eqref{mutacija}. Let $M_{1}(\epsilon,\tau) \geq M(\epsilon,\tau)+M'(\epsilon,\tau)+|\tau|.$ 
Arguing as in the proof of Theorem \ref{milenko}, the continuity of function $\phi(\cdot)$ at the point $t=0$ implies that there exists a finite real number $M_{3}(\epsilon,\tau) \geq M_{1}(\epsilon,\tau)$ such that
\begin{align}\label{kozmetika}
\varphi\bigl(2 c_{B}\bigr)\frac{1}{2} \phi\Biggl(\int_{|\sigma| \leq M_{2}(\epsilon,\tau)}|h({\bf t}-\sigma)| \, d\sigma \Biggr) \leq \epsilon g(\epsilon).
\end{align}
Keeping in mind \eqref{mutacija} and the definition of norm in $L^{p(\cdot)}(\Omega),$ with $\lambda=\epsilon/\textsf{F}_{1}({\bf t},\epsilon,\tau)$ and the meaning clear,
it suffices to show that, for every fixed element $x\in B$ and for every fixed point ${\bf t}\in {\mathbb R}^{n}$ with $|{\bf t}| \geq M_{4}(\epsilon,\tau)\equiv M_{3}(\epsilon,\tau)+|\tau|,$ we have:
\begin{align}
\notag &\phi\Bigl( \bigl\| (h\ast F)({\bf t}+{\bf u}+\tau;x)-c(h\ast F)({\bf t}+{\bf u};x) \bigr\|_{Y} \Bigr)
\\\label{newlab}& =\phi\Biggl( \Bigl\| \int_{{\mathbb R}^{n}}h({\bf t}-\sigma) \cdot \Bigl[ F(\sigma +{\bf u}+\tau;x)-cF(\sigma +{\bf u};x)\Bigr]\, d\sigma \Bigr\|_{Y} \Biggr)
\\\label{newlab1} & \leq \epsilon \varphi(2)\sum_{k\in {\mathbb Z}^{n}}\frac{a_{k}\varphi(a_{k}^{-1})\bigl[ \varphi(|h({\bf t}-\sigma-k)|) \bigr]_{L^{q(\sigma)}(\Omega)}}{\textsf{F}({\bf u}+k,\epsilon,\tau)}+\epsilon g(\epsilon).
\end{align}
Towards this end, observe first that there exists a finite constant $c_{B}>0$ such that (see \eqref{newlab}): {\scriptsize
\begin{align*}
&\phi\Biggl( \Bigl\| \int_{{\mathbb R}^{n}}h({\bf t}-\sigma) \cdot \Bigl[ F(\sigma +{\bf u}+\tau;x)-cF(\sigma +{\bf u};x)\Bigr]\, d\sigma \Bigr\|_{Y} \Biggr)
\\& \leq \phi \Biggl(2 \frac{1}{2}\int_{|\sigma| \geq M_{4}(\epsilon,\tau)}|h({\bf t}-\sigma)| \Bigl\| F({\bf t}+\sigma +\tau ;x)-cF({\bf u}+\sigma)\Bigr\|_{Y}\, d\sigma + \frac{c_{B}}{2}\int_{|\sigma| \leq M_{4}(\epsilon,\tau)}|h({\bf t}-\sigma)| \, d\sigma \Biggr)
\\& \leq \varphi(2)\frac{1}{2}\phi\Biggl(\int_{|\sigma| \geq M_{4}(\epsilon,\tau)} |h({\bf t}-\sigma)|\Bigl\| F({\bf t}+\sigma +\tau ;x)-cF({\bf u}+\sigma)\Bigr\|_{Y}\, d\sigma\Biggr)
\\&+ \varphi\bigl(2 c_{B}\bigr)\frac{1}{2} \phi\Biggl(\int_{|\sigma| \leq M_{4}(\epsilon,\tau)}|h({\bf t}-\sigma)| \, d\sigma \Biggr).
\end{align*}}
Then \eqref{newlab1} follows from the last estimate, \eqref{kozmetika} and the next computation involving the Jensen inequality:{\small
\begin{align*}
 &\phi\Biggl(\int_{|\sigma| \geq M_{4}(\epsilon,\tau)} |h({\bf t}-\sigma)|\Bigl\| F({\bf t}+\sigma +\tau ;x)-cF({\bf u}+\sigma)\Bigr\|_{Y}\, d\sigma\Biggr)   
 \\& =\phi\Biggl(\sum_{k\in {\mathbb Z}^{n}}a_{k}\int_{\sigma \in k+\Omega ; |\sigma| \geq M_{4}(\epsilon,\tau)}a_{k}^{-1} |h({\bf t}-\sigma)|\Bigl\| F({\bf t}+\sigma +\tau ;x)-cF({\bf u}+\sigma)\Bigr\|_{Y}\, d\sigma\Biggr)   
 \\& \leq \sum_{k\in {\mathbb Z}^{n}}a_{k}\varphi \bigl(a_{k}^{-1} \bigr)
 \int_{\sigma \in \Omega ; |\sigma +k| \geq M_{4}(\epsilon,\tau)}\varphi(|h({\bf t}-\sigma)|)\Bigl\| F({\bf t}+\sigma +\tau ;x)-cF({\bf u}+\sigma)\Bigr\|_{Y}\, d\sigma\Biggr),
\end{align*}}
and a simple application of the 
H\"older inequality after that.
\end{proof}

\section{Multi-dimensional Weyl $c$-almost periodic type functions}\label{multiWeylce}

In this section, we will introduce and analyze the multi-dimensional Weyl  $c$-almost periodic type functions following our approach obeyed in \cite[Definition 2.4-Deifnition 2.6]{rmjm},
with $c=1$ (see also \cite[Definition 2.1-Definition 2.3]{rmjm}); we will always assume that
the following condition holds:
\begin{itemize}
\item[(WM2):]
$\emptyset \neq \Lambda \subseteq {\mathbb R}^{n},$ $\emptyset \neq \Lambda' \subseteq {\mathbb R}^{n},$ 
$\emptyset \neq \Omega \subseteq {\mathbb R}^{n}$ is a Lebesgue measurable set such that $m(\Omega)>0,$ $p\in {\mathcal P}(\Omega),$ 
$\Lambda' +\Lambda+ l\Omega\subseteq \Lambda,$ $\Lambda+ l\Omega \subseteq \Lambda$ for all $l>0,$
$\phi : [0,\infty) \rightarrow [0,\infty)$ and ${\mathbb F}: (0,\infty) \times \Lambda \rightarrow (0,\infty).$
\end{itemize} 

\begin{defn}
\label{marinavistce}
\begin{itemize}
\item[(i)]
By $e-W^{[p({\bf u}),\phi,{\mathbb F},c]}_{\Omega,\Lambda',{\mathcal B}}(\Lambda\times X :Y)$ we denote the set consisting of all functions $F : \Lambda \times X \rightarrow Y$ such that, for every $\epsilon>0$ and $B\in {\mathcal B},$ there exist two finite real numbers
$l>0$
and
$L>0$ such that for each ${\bf t}_{0}\in \Lambda'$ there exists $\tau \in B({\bf t}_{0},L)\cap \Lambda'$ such that
\begin{align*}
\sup_{x\in B}\sup_{{\bf t}\in \Lambda}l^{n}{\mathbb F}(l,{\bf t})\phi\Bigl( \bigl\| F({\bf t}+{\bf \tau}+l{\bf u};x)-cF({\bf t}+l{\bf u};x) \bigr\|_{Y}\Bigr)_{L^{p({\bf u})}(\Omega)} <\epsilon.
\end{align*}
\index{space!$e-W^{[p({\bf u}),\phi,{\mathbb F},c]}_{\Omega,\Lambda',{\mathcal B}}(\Lambda\times X :Y)$}
\item[(ii)] By $W^{[p({\bf u}),\phi,{\mathbb F},c]}_{\Omega,\Lambda',{\mathcal B}}(\Lambda\times X :Y)$ we denote the set consisting of all functions $F : \Lambda \times X \rightarrow Y$ such that, for every $\epsilon>0$ and $B\in {\mathcal B},$ there exists a finite real number
$L>0$ such that for each ${\bf t}_{0}\in \Lambda'$ there exists $\tau \in B({\bf t}_{0},L) \cap \Lambda'$ such that
\begin{align*}
\limsup_{l\rightarrow +\infty}\sup_{x\in B}\sup_{{\bf t}\in \Lambda}l^{n}{\mathbb F}(l,{\bf t})\phi\Bigl( \bigl\| F({\bf t}+{\bf \tau}+l{\bf u};x)-cF({\bf t}+l{\bf u};x) \bigr\|_{Y}\Bigr)_{L^{p({\bf u})}(\Omega:Y)} 
<\epsilon.
\end{align*}
\index{space!$W^{[p({\bf u}),\phi,{\mathbb F},c]}_{\Omega,\Lambda',{\mathcal B}}(\Lambda\times X :Y)$}
\end{itemize}
\end{defn}

\begin{defn}
\label{marinavis1tce}
\begin{itemize}
\item[(i)]
By $e-W^{[p({\bf u}),\phi,{\mathbb F},c]_{1}}_{\Omega,\Lambda',{\mathcal B}}(\Lambda\times X :Y)$ we denote the set consisting of all functions $F : \Lambda \times X \rightarrow Y$ such that, for every $\epsilon>0$ and $B\in {\mathcal B},$ there exist two finite real numbers
$l>0$
and
$L>0$ such that for each ${\bf t}_{0}\in \Lambda'$ there exists $\tau \in B({\bf t}_{0},L)\cap \Lambda'$ such that
\begin{align*}
\sup_{x\in B}\sup_{{\bf t}\in \Lambda}l^{n}{\mathbb F}(l,{\bf t})\phi\Bigl( \bigl\| F({\bf t}+{\bf \tau}+l{\bf u};x)-cF({\bf t}+l{\bf u};x) \bigr\|_{L^{p({\bf u})}(\Omega:Y)} \Bigr)
<\epsilon.
\end{align*}
\index{space!$e-W^{[p({\bf u}),\phi,{\mathbb F},c]_{1}}_{\Omega,\Lambda',{\mathcal B}}(\Lambda\times X :Y)$}
\item[(ii)] By $W^{[p({\bf u}),\phi,{\mathbb F},c]_{1}}_{\Omega,\Lambda',{\mathcal B}}(\Lambda\times X :Y)$ we denote the set consisting of all functions $F : \Lambda \times X \rightarrow Y$ such that, for every $\epsilon>0$ and $B\in {\mathcal B},$ there exists a finite real number
$L>0$ such that for each ${\bf t}_{0}\in \Lambda'$ there exists $\tau \in B({\bf t}_{0},L) \cap \Lambda'$ such that
\begin{align*}
\limsup_{l\rightarrow +\infty}\sup_{x\in B}\sup_{{\bf t}\in \Lambda}l^{n}{\mathbb F}(l,{\bf t})\phi\Bigl( \bigl\| F({\bf t}+{\bf \tau}+{\bf u};x)-cF({\bf t}+{\bf u};x) \bigr\|_{L^{p({\bf u})}(l\Omega:Y)} \Bigr)
<\epsilon.
\end{align*}
\index{space!$W^{[p({\bf u}),\phi,{\mathbb F},c]_{1}}_{\Omega,\Lambda',{\mathcal B}}(\Lambda\times X :Y)$}
\end{itemize}
\end{defn}

\begin{defn}
\label{marinavis2tce}
\begin{itemize}
\item[(i)]
By $e-W^{[p({\bf u}),\phi,{\mathbb F},c]_{2}}_{\Omega,\Lambda',{\mathcal B}}(\Lambda\times X :Y)$ we denote the set consisting of all functions $F : \Lambda \times X \rightarrow Y$ such that, for every $\epsilon>0$ and $B\in {\mathcal B},$ there exist two finite real numbers
$l>0$
and
$L>0$ such that for each ${\bf t}_{0}\in \Lambda'$ there exists $\tau \in B({\bf t}_{0},L)\cap \Lambda'$ such that
\begin{align*}
\sup_{x\in B}\sup_{{\bf t}\in \Lambda}\phi\Bigl( l^{n}{\mathbb F}(l,{\bf t}) \bigl\| F({\bf t}+{\bf \tau}+l{\bf u};x)-cF({\bf t}+l{\bf u};x) \bigr\|_{L^{p({\bf u})}(\Omega:Y)} \Bigr)
<\epsilon.
\end{align*}
\index{space!$e-W^{[p({\bf u}),\phi,{\mathbb F},c]_{2}}_{\Omega,\Lambda',{\mathcal B}}(\Lambda\times X :Y)$}
\item[(ii)] By $W^{[p({\bf u}),\phi,{\mathbb F},c]_{2}}_{\Omega,\Lambda',{\mathcal B}}(\Lambda\times X :Y)$ we denote the set consisting of all functions $F : \Lambda \times X \rightarrow Y$ such that, for every $\epsilon>0$ and $B\in {\mathcal B},$ there exists a finite real number
$L>0$ such that for each ${\bf t}_{0}\in \Lambda'$ there exists $\tau \in B({\bf t}_{0},L) \cap \Lambda'$ such that
\begin{align*}
\limsup_{l\rightarrow +\infty}\sup_{x\in B}\sup_{{\bf t}\in \Lambda}\phi\Bigl( l^{n}{\mathbb F}(l,{\bf t})\bigl\| F({\bf t}+{\bf \tau}+l{\bf u};x)-cF({\bf t}+l{\bf u};x) \bigr\|_{L^{p({\bf u})}(\Omega:Y)} \Bigr)
<\epsilon.
\end{align*}
\index{space!$W^{[p({\bf u}),\phi,{\mathbb F},c]_{2}}_{\Omega,\Lambda',{\mathcal B}}(\Lambda\times X :Y)$}
\end{itemize}
\end{defn}

It is clear that the notion from the second parts of the above definitions extends the corresponding notion from the first parts of these definitions. Using the Jensen integral inequality, we can clarify certain embedding results between the introduced spaces, provided that the function $\phi(\cdot)$ is convex (concave); see \cite{rmjm, weyl-varible} for more details. The statement of Proposition \ref{piksi} can be formulated for some classes of functions introduced in the fourth section and the above three definitions; this could be also left to the interested readers to make precise.

In many concrete situations, the situation in which $\Lambda'\neq \Lambda$ can occur:

\begin{example}\label{mesecina}
(see e.g., \cite[Example 2.12(i)]{multi-ce} and \cite[Example 3.1]{rmjm})
Suppose that 
the complex-valued mapping $t\mapsto g_{j}(s)\, ds,$ $t\in {\mathbb R}$ is essentially bounded and 
(equi-)Weyl-$(p,c)$-almost periodic
($1\leq j \leq n$). Define
\begin{align*}
F\bigl(t_{1},\cdot \cdot \cdot,t_{2n}\bigr):=\prod_{j=1}^{n}\Bigl[g_{j}\bigl(t_{j+n}\bigr)-g_{j}\bigl(t_{j}\bigr)\Bigr],\ \mbox{ where } t_{j}\in {\mathbb R} \mbox{  for }\ 1\leq j\leq 2n,
\end{align*}
and $\Lambda':=\{({\bf \tau},{\bf \tau}) : {\bf \tau} \in {\mathbb R^{n}} \}.$ 
Arguing similarly is \cite[Example 2.13(ii)]{marko-manuel-ap}, we can show that 
the function $F(\cdot)$ 
belongs to the class 
$(e-)W^{[p,x,l^{-n/p},c]}_{[0,1]^{n},\Lambda'}({\mathbb R}^{2n} : {\mathbb C}).$ 
\end{example}

It is clear that all introduced spaces are invariant under the pointwise multiplications with complex scalars provided that condition (F) holds.
The translation invariance of spaces introduced in Definition \ref{marinavistce} and Definition \ref{marinavis1tce} holds provided that, for every $\tau \in \Lambda,$ we have
$$
\sup_{l>0,{\bf t}\in \Lambda}\frac{{\mathbb F}(l,{\bf t})}{{\mathbb F}(l,{\bf t}+\tau)}<+\infty,
$$
while the translation invariance of spaces introduced in Definition \ref{marinavis2tce} holds provided this condition and condition (F).
The interested reader may try to formulate sufficient conditions which ensure that the introduced function spaces are invariant, in a certain sense, under the operations of form $F(\cdot;\cdot) \mapsto F(b \cdot ; b' \cdot),$ where $b>0$ and $b' \in {\mathbb C} \setminus \{0\}$ (see also the item (iii) in the paragraph following \cite[Example 2.8]{rmjm}). Furthermore, it can be simply shown that for any scalar-valued function $F(\cdot;\cdot)$ which is bounded away from zero on elements of the collection ${\mathcal B},$ the function $1/F(\cdot;\cdot)$ is well defined and belongs to the same space of functions as $F(\cdot;\cdot),$ with the constant $c$ replaced by $1/c$ in the corresponding space and the meaning clear (see also \cite[Theorem 2.1]{BIMV}). 

The conclusions from the following result can be also formulated for the classes of functions introduced in Definition \ref{marinavis1tce} and Definition \ref{marinavis2tce} (cf. \cite[Proposition 2.1-Proposition 2.2]{BIMV} for the one-dimensional case):

\begin{prop}
\begin{itemize}
\item[(i)] Suppose that 
the function $\phi(\cdot)$ is monotonically increasing
and
$F\in (e-)W^{[p({\bf u}),\phi,{\mathbb F},c]}_{\Omega,\Lambda',{\mathcal B}}(\Lambda\times X :Y)$. Then we have
$\| F (\cdot;\cdot)\|_{Y}\in (e-)W^{[p({\bf u}),\phi,{\mathbb F},|c|]}_{\Omega,\Lambda',{\mathcal B}}(\Lambda\times X :Y)$.
\item[(ii)] Suppose that 
$F\in (e-)W^{[p({\bf u}),\phi,{\mathbb F},c]}_{\Omega,\Lambda',{\mathcal B}}(\Lambda\times X :Y)$. Then we have $\check{F}\in (e-)W^{[p_{1}({\bf u}),\phi,{\mathbb F}_{1},c]}_{-\Omega,-\Lambda',{\mathcal B}}((-\Lambda) \times X :Y),$ where $p_{1}(\cdot):=p(-\cdot)$ and ${\mathbb F}_{1}(\cdot;\cdot):={\mathbb F}(\cdot; -\cdot).$ 
\end{itemize}
\end{prop}

\begin{proof}
The proof of (i) simply follows from Lemma \ref{aux}(iii), our assumption that the function $\phi(\cdot)$ is monotonically increasing
and the inequality 
$$
\Bigl| \bigl\| F({\bf t}+{\bf \tau}+l{\bf u};x)\bigr\|_{Y}-|c| \bigl\|F({\bf t}+l{\bf u};x) \bigr\|_{Y}\Bigr| \leq \bigl\| F({\bf t}+{\bf \tau}+l{\bf u};x)-cF({\bf t}+l{\bf u};x) \bigr\|_{Y},
$$
with the notation and meaning clear. The proof of (ii) follows from the chain rule, the definition of norm in $L^{p_{1}(\cdot)}(-\Omega)$ and the next equalities:{\scriptsize
\begin{align*}
& l^{n}{\mathbb F}(l,-{\bf t})\Bigl[\phi\bigl( \bigl\| F(-{\bf t}-\tau-l{\bf u};x)-cF(-{\bf t}-l{\bf u};x)\bigr\|_{Y} \bigr)\Bigr]_{L^{p_{1}(\cdot)}(-\Omega )}
\\ & =l^{n}{\mathbb F}(l,-{\bf t})\inf \Biggl\{ \lambda >0 : \int_{-\Omega}\varphi_{p(-{\bf u})}\Biggl( \frac{\phi\bigl( \| F(-{\bf t}-\tau-l{\bf u};x)-cF(-{\bf t}-l{\bf u};x)\|_{Y} \bigr)}{\lambda}\Biggr)\, d{\bf u}\leq 1 \Biggr\} 
\\& =l^{n}{\mathbb F}(l,-{\bf t})\inf \Biggl\{ \lambda >0 : \int_{\Omega}\varphi_{p({\bf u})}\Biggl( \frac{\phi\bigl( \| F(-{\bf t}-\tau+l{\bf u};x)-cF(-{\bf t}+l{\bf u};x)\|_{Y} \bigr)}{\lambda}\Biggr)\, d{\bf u}\leq 1 \Biggr\} ,
\end{align*}}
with the notation and meaning clear.
\end{proof}

In what follows, we will extend the statements of \cite[Proposition 2.3, Corollary 2.1, Proposition 2.4]{BIMV} to the multi-dimensional setting:

\begin{thm}\label{kokeza}
Suppose that the function ${\mathbb F}(\cdot;\cdot)$ does not depend on the second argument.
\begin{itemize}
\item[(i)] Suppose that $m\in {\mathbb N},$ $j\Lambda' +\Lambda+ l\Omega\subseteq \Lambda$ for all $l\geq 0$ and $j\in {\mathbb N},$ as well as that condition \emph{(F)} holds and there exists a finite real constant $c_{m}>0$ such that
\begin{align}\label{dalije}
\phi\bigl(x_{1}+\cdot \cdot \cdot +x_{m}\bigr) \leq c_{m}\Bigl[\phi \bigl(x_{1}\bigr) +\cdot \cdot \cdot +\phi\bigl(x_{m}\bigr) \Bigr],\quad x_{i}\geq 0\ \ \bigl( i\in {\mathbb N}_{m} \bigr).
\end{align}
Suppose, further, that $F\in (e-)W^{[p({\bf u}),\phi,{\mathbb F},c]}_{\Omega,\Lambda',{\mathcal B}}(\Lambda\times X :Y)$, resp. $F\in (e-)W^{[p({\bf u}),\phi,{\mathbb F},c]_{i}}_{\Omega,\Lambda',{\mathcal B}}(\Lambda\times X :Y)$ for $i=1,2.$ Then $F\in (e-)W^{[p({\bf u}),\phi,{\mathbb F},c^{m}]}_{\Omega,m\Lambda',{\mathcal B}}(\Lambda\times X :Y)$, resp. $F\in (e-)W^{[p({\bf u}),\phi,{\mathbb F},c^{m}]_{i}}_{\Omega,m\Lambda',{\mathcal B}}(\Lambda\times X :Y)$ provided $i=1,2$ and the function $\phi(\cdot)$ is monotonically increasing. 
\item[(ii)] Suppose that $m\in 2{\mathbb Z} \setminus \{0\},$
$p\in {\mathbb N},$ $(m, n) = 1,$ $|c| = 1$ and $\arg(c) = \pi m/p$ [$m\in 2{\mathbb Z}+1 ,$
$p\in {\mathbb N},$ $(m, n) = 1,$ $|c| = 1$ and $\arg(c) = \pi m/p$],
$m\in {\mathbb N},$ $j\Lambda' +\Lambda+ l\Omega\subseteq \Lambda$ for all $l\geq 0$ and $j\in {\mathbb N},$ as well as that condition \emph{(F)} holds and there exists a finite real constant $c_{m}>0$ such that \eqref{dalije} holds. If $F\in (e-)W^{[p({\bf u}),\phi,{\mathbb F},c]}_{\Omega,\Lambda',{\mathcal B}}(\Lambda\times X :Y)$, resp. $F\in (e-)W^{[p({\bf u}),\phi,{\mathbb F},c]_{i}}_{\Omega,\Lambda',{\mathcal B}}(\Lambda\times X :Y)$ for $i=1,2,$ then $F\in (e-)W^{[p({\bf u}),\phi,{\mathbb F},1]}_{\Omega,m\Lambda',{\mathcal B}}(\Lambda\times X :Y)$ [$F\in (e-)W^{[p({\bf u}),\phi,{\mathbb F},-1]}_{\Omega,m\Lambda',{\mathcal B}}(\Lambda\times X :Y)$], resp. $F\in (e-)W^{[p({\bf u}),\phi,{\mathbb F},1]_{i}}_{\Omega,m\Lambda',{\mathcal B}}(\Lambda\times X :Y)$ [$F\in (e-)W^{[p({\bf u}),\phi,{\mathbb F},-1]_{i}}_{\Omega,m\Lambda',{\mathcal B}}(\Lambda\times X :Y)$], provided $i=1,2$ and the function $\phi(\cdot)$ is monotonically increasing. 
\item[(iii)] Suppose that
$|c| = 1,$ $\arg(c)\notin \pi {\mathbb Q},$
$j\Lambda' +\Lambda+ l\Omega\subseteq \Lambda$ for all $l\geq 0$ and $j\in {\mathbb N},$ as well as that condition \emph{(F)} holds and for each $m\in {\mathbb N}$ there exists a finite real constant $c_{m}>0$ such that \eqref{dalije} holds. Let the function $\varphi(\cdot)$ be continuous at zero.
Suppose, further, that $F\in (e-)W^{[p({\bf u}),\phi,{\mathbb F},c]}_{\Omega,\Lambda',{\mathcal B}}(\Lambda\times X :Y)$, resp. $F\in (e-)W^{[p({\bf u}),\phi,{\mathbb F},c]_{i}}_{\Omega,\Lambda',{\mathcal B}}(\Lambda\times X :Y)$ for $i=1,2.$ Then $F\in (e-)W^{[p({\bf u}),\phi,{\mathbb F},c']}_{\Omega,\Lambda',{\mathcal B}}(\Lambda\times X :Y)$, provided that for each set $B\in {\mathcal B}$ the following condition holds
\begin{align}\label{drvane}
\sup_{l>1,{\bf t}\in \Lambda; x\in B} l^{n}{\mathbb F}(l)\Bigl[\phi\bigl( \| F({\bf t}+l{\bf u};x) \|_{Y} \bigr)\Bigr]_{L^{p({\bf u})}(\Omega)}<+\infty ,
\end{align}
resp. $F\in (e-)W^{[p({\bf u}),\phi,{\mathbb F},c']_{1}}_{\Omega,\Lambda',{\mathcal B}}(\Lambda\times X :Y)$ [$F\in (e-)W^{[p({\bf u}),\phi,{\mathbb F},c']_{2}}_{\Omega,\Lambda',{\mathcal B}}(\Lambda\times X :Y)$] provided that the function $\phi(\cdot)$ is monotonically increasing and
for each set $B\in {\mathcal B}$ the following condition holds
\begin{align*}
\sup_{l>1,{\bf t}\in \Lambda; x\in B} l^{n}{\mathbb F}(l)\phi\Bigl( \| F({\bf t}+l{\bf u};x)\|_{L^{p({\bf u})}(\Omega : Y)}\Bigr)<+\infty \ \ 
\end{align*}
\begin{align*}
\Biggl[ \ \ \sup_{l>1,{\bf t}\in \Lambda; x\in B} \phi\Bigl( l^{n}{\mathbb F}(l)\| F({\bf t}+l{\bf u};x)\|_{L^{p({\bf u})}(\Omega : Y)}\Bigr)<+\infty \ \ \Biggr].
\end{align*}
\end{itemize}
\end{thm}

\begin{proof}
We will prove the statements (i) and (iii) for the class $(e-)W^{[p({\bf u}),\phi,{\mathbb F},c]}_{\Omega,\Lambda',{\mathcal B}}(\Lambda\times X :Y),$ only. Clearly, we have the following decomposition (${\bf t}\in \Lambda;$ ${\bf u}\in \Omega$; $l>0$):
\begin{align*}
F({\bf t}&+m\tau +l{\bf u} ;x) -c^{m}F({\bf t}+l{\bf u};x)
\\& =\sum_{j=0}^{m-1}c^{j}\Bigl[ F({\bf t}+(m-j)\tau +l{\bf u};x ) -c^{m}F({\bf t}+(m-j-1)+l{\bf u};x) \Bigr].
\end{align*}
Therefore, our assumptions imply {\small
\begin{align*}
& \phi \Bigl(\bigl\|F({\bf t}+m\tau +l{\bf u};x ) -c^{m}F({\bf t}+l{\bf u};x)\bigr\|_{Y}\Bigr)_{L^{p({\bf u})}(\Omega)}
\\& \leq c_{m}\sum_{j=0}^{m-1}\varphi\bigl(c^{j} \bigr) \phi \Bigl(\bigl\|F({\bf t}+(m-j)\tau +l{\bf u};x ) -c^{m}F({\bf t}+(m-j-1)\tau+l{\bf u};x)\bigr\|_{Y}\Bigr)_{L^{p({\bf u})}(\Omega)}
\end{align*}}
and ${\bf t}+(m-j-1)\tau \in \Lambda$ for all ${\bf t}\in \Lambda$ and
$0\leq j\leq m-1.$ The final conclusion of (i) simply follows from the above. To prove (iii), it should be only recalled that the set $\{c^{m} : m\in {\mathbb N}\}$
is dense in the unit sphere $S_{1}\equiv \{z \in {\mathbb C} : |z|=1\}$ so that there exists a strictly increasing sequence $(l_{k})$ of positive integers such that $\lim_{l\rightarrow +\infty}c^{l_{k}}=c'.$ Then the conclusion follows similarly as in the proof of
\cite[Proposition 2.3]{BIMV}, by applying the first part of this theorem, our assumption with $m=2$ and the estimate \eqref{drvane}.
\end{proof}

In what follows, we will revisit once more \cite[Example 2.1-Example 2.2]{BIMV} and \cite[Example 2.7-Example 2.8]{rmjm}:

\begin{example}\label{danikali}
Let $\Omega =[0,1]^{n}.$
\begin{itemize}
\item[(i)] Suppose that $\emptyset \neq K\subseteq {\mathbb R}^{n}$ and $F({\bf t}):=\chi_{K}({\bf t}),$ ${\bf t}\in {\mathbb R}^{n}.$
We will prove that for each $p\in D_{+}(\Omega)$ and $c\in {\mathbb C} \setminus \{0\}$ we have $F\in e-W^{[p({\bf u}),x,l^{-\sigma},c]}_{\Omega,{\mathbb R}^{n} }({\mathbb R}^{n} : {\mathbb C}).$ Keeping in mind Lemma \ref{aux}(ii), 
we get that (${\bf \tau} \in {\mathbb R}^{n};$ $l>0$):
\begin{align*}
\sup_{{\bf t}\in {\mathbb R}^{n}}& l^{n-\sigma}\Bigl\| \chi_{K}({\bf t}+\tau +l{\bf u})-c\chi_{K}({\bf t}+l{\bf u}) \Bigr\|_{L^{p({\bf u})}(\Omega)}
\\ & \leq 4 \sup_{{\bf t}\in {\mathbb R}^{n}} l^{n-\sigma}\Bigl\| \chi_{K}({\bf t}+\tau +l{\bf u})-c\chi_{K}({\bf t}+l{\bf u}) \Bigr\|_{L^{p^{+}}(\Omega)}
\\ & = 4 \sup_{{\bf t}\in {\mathbb R}^{n}}l^{-\sigma}\Bigl\| \chi_{K}({\bf t}+\tau +{\bf u})-c\chi_{K}({\bf t}+{\bf u}) \Bigr\|_{L^{p^{+}}(l\Omega)}
\\& \leq 4\sup_{{\bf t}\in {\mathbb R}^{n}}l^{-\sigma}\Biggl[ \bigl\| \chi_{K}(\cdot)\bigr\|_{L^{p^{+}}(l\Omega \cap [K-{\bf t}-{\bf \tau}])}+ |c|\bigl\| \chi_{K}(\cdot)\bigr\|_{L^{p^{+}}(l\Omega \cap [K-{\bf t}])} \Biggr]
\\& \leq 4l^{-\sigma}(1+|c|)m(K).
\end{align*}
This simply implies the required.
\item[(ii)] Set $F({\bf t}):=\chi_{[0,\infty)^{n}}({\bf t}),$ ${\bf t}\in {\mathbb R}^{n}.$ In \cite[Example 2.8]{rmjm}, we have proved that $F\in W^{[p,x,l^{-\sigma},1]}_{\Omega,{\mathbb R}^{n} }({\mathbb R}^{n} : {\mathbb C})$ if and only if $\sigma>(n-1)/p,$ as well as that
there is no $\sigma>0$ such that $F\in e-W^{[p,x,l^{-\sigma},1]}_{\Omega,{\mathbb R}^{n} }({\mathbb R}^{n} : {\mathbb C});$ similarly, we have that there is no $\sigma >0$ and $c\in {\mathbb C} \setminus \{0\}$ such that $F\in e-W^{[p,x,l^{-\sigma},c]}_{\Omega,{\mathbb R}^{n} }({\mathbb R}^{n} : {\mathbb C}).$ Since
\begin{align*}
\sup_{{\bf t}\in {\mathbb R}^{n}} \Bigl\| \chi_{[0,\infty)^{n}}({\bf t}+\tau +l{\bf u})-c\chi_{[0,\infty)^{n}}({\bf t}+l{\bf u}) \Bigr\|_{L^{p}(\Omega)} \geq |1-c|,
\end{align*}
as easily approved, we get that there is no $c\in {\mathbb C} \setminus \{0,1\}$ such that $F\in W^{[p,x,l^{-\sigma},c]}_{\Omega,{\mathbb R}^{n} }({\mathbb R}^{n} : {\mathbb C})$ for $n\geq \sigma >(n-1)/p.$ This is also the optimal result we can obtain because for any $\sigma>0$ and any essentially bounded function $F(\cdot)$ we have $F\in e-W^{[p,x,l^{-\sigma},c]}_{\Omega,{\mathbb R}^{n} }({\mathbb R}^{n} : {\mathbb C}).$
\end{itemize}
\end{example}

Regarding the convolution invariance of spaces introduced in this section, we will clarify just one result for the class $(e-)W^{[p({\bf u}),\phi,{\mathbb F},c]}_{\Omega,\Lambda',{\mathcal B}}(\Lambda\times X :Y);$ the proof is almost the same as the proof of \cite[Theorem 2.9]{rmjm} and therefore omitted:

\begin{thm}\label{shokiran1ce}
Suppose that 
$\phi :[0,\infty) \rightarrow [0,\infty) $ is a convex monotonically increasing function satisfying condition \emph{(F)}.
Suppose, further, that
$h\in L^{1}({\mathbb R}^{n}),$ $\Omega=[0,1]^{n} $, $F\in (e-)W^{[p({\bf u}),\phi,{\mathbb F},c]}_{\Omega,\Lambda',{\mathcal B}}({\mathbb R}^{n}\times X :Y),$ $1/p({\bf u})+1/q({\bf u})=1,$ and for each $x\in X$ we have $\sup_{{\bf t}\in {\mathbb R}^{n}}\| F({\bf t};x)\|_{Y}<\infty.$ If ${\mathbb F}_{1} : (0,\infty) \times {\mathbb R}^{n} \rightarrow (0,\infty),$ $p_{1}\in {\mathcal P}({\mathbb R^{n}})$ and if, for every ${\bf t}\in {\mathbb R}^{n}$ and $l>0,$  there exists a sequence $(a_{k})_{k\in  l{\mathbb Z}^{d}}$
of positive real numbers such that $\sum_{k\in  l{\mathbb Z}^{n}}a_{k}=1$ and {\small
\begin{align*}
\int_{\Omega}\varphi_{p_{1}({\bf u})}\Biggl( 2\sum_{k\in l{\mathbb Z}^{n}}a_{k}l^{-n}\Bigl[\varphi\bigl( a_{k}^{-1}l^{n}h(k-l{\bf v}) \bigr)\Bigr]_{L^{q({\bf v})}(\Omega)}{\mathbb F}_{1}(l,{\bf t})
\bigl[{\mathbb F}(l,{\bf t}+l{\bf u}-k)\bigr]^{-1}
\Biggr)\, d{\bf u} \leq 1,
\end{align*}}
then $h\ast F\in (e-)W^{[p_{1}({\bf u}),\phi,{\mathbb F}_{1},c]}_{\Omega,\Lambda',{\mathcal B}}({\mathbb R}^{n}\times X :Y).$
\end{thm}

If $p\in [1,\infty),$ then any Stepanov $(p,c)$-quasi-asymptotically almost periodic function is Weyl-$(p,c)$-almost periodic (see \cite[Proposition 3.3]{BIMV}), which also holds for the corresponding classes of uniformly recurrent functions. In the one-dimensional setting, the generalized Weyl uniform recurrence in Lebesgue spaces with variable exponents has been thoroughly analyzed in \cite[Section 2]{dumath2}. This notion can be also introduced and analyzed in the multi-dimensional setting; for the sake of brevity, we will only mention the following notion here: Let (WM2) hold. Then we say that a function $F : \Lambda \times X \rightarrow Y$ is Weyl-$[\Omega,{\mathcal B},\Lambda', p, \phi,{\bf F},c]$-uniformly recurrent if and only if 
for each set $B\in {\mathcal B}$ we can find a sequence $(\tau_{k})$ in $\Lambda'$ such that $\lim_{k\rightarrow +\infty}|\tau_{k}|=+\infty$ as well as that
$$
\lim_{k\rightarrow +\infty}\limsup_{l\rightarrow +\infty}\sup_{{\bf t}\in \Lambda;x\in B}\Biggl[ {\bf F}(l,{\bf t})\phi\Bigl(\| F(\cdot l+{\bf t}+\tau_{k};x)-cF(\cdot l+{\bf t};x)\|_{Y}\Bigr)_{L^{p(\cdot)}(\Omega)} \Biggr]=0.
$$
The above-mentioned result about the set-theoretical embedding of space of Stepanov $(p,c)$-quasi-asymptotically almost periodic functions into the space of Weyl-$(p,c)$-almost periodic functions can be generalized in many different directions; in \cite[Proposition 6]{dumath2}, e.g., we have shown that any Stepanov-$(p,\phi,F)$-quasi-asymptotically uniformly recurrent function is Weyl-$(p(x),\phi,F_{1})$-uniformly recurrent under certain assumptions. 
This result can be formulated in the multi-dimensional setting but we will consider here only the constant coefficient case $p(\cdot) \equiv p\in [1,\infty)$ for brevity:

\begin{prop}\label{radosbajic}
Suppose that 
\emph{$(MD-B)_{S}$} holds
and a function $F : \Lambda \times X \rightarrow Y$ is  Stepanov-$[\Omega, {\mathcal B}, \Lambda',\Lambda, p,\phi,{\bf F},c]$-quasi-asymptotically uniformly recurrent. If
${\bf F}_{1} : (0,\infty) \times \Lambda \rightarrow (0,\infty)$ satisfies
\begin{equation*}
\lim_{k\rightarrow +\infty} \limsup_{l\rightarrow +\infty}\sup_{{\bf t}\in \Lambda}{\bf F}_{1}(l,{\bf t})\sum_{a\in {\mathbb Z}^{n} \cap [0,l]^{n}}\frac{1}{F({\bf t}+a,k)}<\infty
\end{equation*}
and
\begin{equation*}
\lim_{l\rightarrow +\infty}\sup_{{\bf t}\in \Lambda}{\bf F}_{1}(l,{\bf t})=0,
\end{equation*}
then
the function $F(\cdot ;\cdot)$ is 
Weyl-$[\Omega,{\mathcal B},\Lambda', p, \phi,{\bf F},c]$-uniformly recurrent.
\end{prop}

We close this section with the observation that the notion introduced in \cite[Definition 3.17]{rmjm} can be also analyzed following the approach employed here, by replacing the corresponding differences $\| \cdot -\cdot \|$ in definitions with a general constant $c\in {\mathbb C} \setminus \{0\}$ and the differences $\| \cdot -c \cdot \|$.

\section{Applications to the abstract Volterra integro-differential equations}\label{apply}

This section is devoted to some applications of our abstract theoretical results to the abstract Volterra integro-differential equations. \vspace{0.1cm}

1. We start by noting that all established applications made in the fourth section of \cite{rmjm}, including applications to the d'Alemebert formula, the Gaussian semigroups in ${\mathbb R}^{n}$ and the nonautonomous differential equations of the first order, can be straightforwardly formulated for the corresponding classes of multi-dimensional (equi-)Weyl $c$-almost periodic type functions considered in this paper. In this part, we will present the following illustrative application of Theorem \ref{milenko2}, only: Let $Y$ be one of the spaces $L^{p}({\mathbb R}^{n}),$ $C_{0}({\mathbb R}^{n})$ or $BUC({\mathbb R}^{n}),$ where $1\leq p<\infty.$ Then the Gaussian semigroup\index{Gaussian semigroup}
$$
(G(t)F)(x):=\bigl( 4\pi t \bigr)^{-(n/2)}\int_{{\mathbb R}^{n}}F(x-y)e^{-\frac{|y|^{2}}{4t}}\, dy,\quad t>0,\ f\in Y,\ x\in {\mathbb R}^{n}
$$
can be extended to a bounded analytic $C_{0}$-semigroup of angle $\pi/2,$ generated by the Laplacian $\Delta_{Y}$ acting with its maximal distributional domain in $Y.$ Suppose now that $t_{0}>0$ is a fixed real number, $\Omega=[0,1]^{n},$ ${\mathbb D}=\Lambda={\mathbb R}^{n}$ and the function $F : {\mathbb R}^{n} \rightarrow {\mathbb C}$ is Stepanov-$[\Omega, \Lambda',{\mathbb R}^{n}, p, x, \textsf{F}, c]$-quasi-asymptotically almost periodic, resp. Stepanov-$[\Omega, \Lambda',{\mathbb R}^{n}, p, x, {\bf F}, c]$-quasi-asymptotically uniformly recurrent. Then the function $x\mapsto (G(t_{0})F)(x),$ $x\in {\mathbb R}^{n}$ is Stepanov-$[\Omega, \Lambda',{\mathbb R}^{n}, p, x, \textsf{F}_{1}, c]$-quasi-asymptotically almost periodic, resp. Stepanov-$[\Omega, \Lambda',{\mathbb R}^{n}, p, x, {\bf F}_{1}, c]$-quasi-
asymptotically uniformly recurrent provided that there exists a continuous function $g : [0,\infty) \rightarrow [0,\infty)$ with $g(0+)=0+$ such that
for each $\epsilon>0$ and $\tau \in \Lambda',$ resp. for each $n\in {\mathbb N}$ and $\tau \in \lambda',$ there exists $M(\epsilon,\tau)>0,$ resp. $M(n,\tau)>0,$ such that for each ${\bf t} \in {\mathbb R}^{n}$ with $|{\bf t}|\geq M(\epsilon,\tau),$ resp. $|{\bf t}|\geq M(n,\tau),$ we have
$$
\int_{[0,1]^{n}}\Biggl[ \textsf{F}_{1}({\bf t}, \epsilon,\tau)\Biggl( \sum_{k\in {\mathbb Z}^{n}}\frac{e^{-|{\bf t}-k|^{2}}}{\textsf{F}({\bf u}+k,\epsilon,\tau)} +g(\epsilon)\Biggr)\Biggr]^{p}\, d{\bf u} \leq 1,
$$
resp.
$$
\int_{[0,1]^{n}}\Biggl[ {\bf F}_{1}({\bf t}, n)\Biggl( \sum_{k\in {\mathbb Z}^{n}}\frac{e^{-|{\bf t}-k|^{2}}}{{\bf F}({\bf u}+k,n)} +g(1/n)\Biggr)\Biggr]^{p}\, d{\bf u} \leq 1.
$$
However, this is a pure theoretical condition which cannot be so simply verified in some practical situations; see also Theorem \ref{milenko} and Theorem \ref{milenko1} which can be also applied here.

2. Concerning the regular solutions of the inhomogeneous wave equations given by the d'Alembert formula,
we would like to note that the analysis carried out in the issues \cite[2.2-2.3, Section 4]{multi-omega-ce} can be also used to justify the introduction of notion in Definition \ref{drasko-presing-12345} and Definition \ref{drasko-presing123456}. More precisely, suppose that 
$\omega \in {\mathbb R} \setminus \{0\},$ $k\in {\mathbb N}$ and $c\in {\mathbb C} \setminus \{0\}$ satisfies 
$c^{k-1}=1.$
Recall that the regular solution of the wave equation $u_{tt}=a^{2}u_{xx}$ in domain $\{(x,t) : x\in {\mathbb R},\ t>0\},$ equipped with the initial conditions $u(x,0)=f(x)\in C^{2}({\mathbb R})$ and $u_{t}(x,0)=g(x)\in C^{1}({\mathbb R}),$ is given by the famous d'Alembert formula
$$
u(x,t)=\frac{1}{2}\bigl[ f(x-at) +f(x+at) \bigr]+\frac{1}{2a}\int^{x+at}_{x-at}g(s)\, ds,\quad x\in {\mathbb R}, \ t>0.
$$  
Suppose that ${\mathbb D}$ is any unbounded set in the plane ${\mathbb R}^{2}$ such that ($g^{[1]}(\cdot)\equiv \int^{\cdot}_{0}g(s)\, ds$):
\begin{align*}
&\lim_{|(x,t)|\rightarrow +\infty, (x,t) \in {\mathbb D}}\Biggl[
|f(x-at+\omega)-cf(x-at)|+\Bigl| g^{[1]}(x-at+\omega)-cg^{[1]}(x-at)\Bigr|
\\& +\sum_{j=1}^{k}\Bigl(| f(x+at+j\omega)-cf(x+at+(j-1)\omega) |
\\& +\Bigl| g^{[1]}(x+at+j\omega)-cg^{[1]}(x+at+(j-1)\omega)\Bigr|\Bigr)
\Biggr]=0,
\end{align*}
as well as that
$$
\omega_{1}:=\frac{1+k}{2}\omega \mbox{  and  }\omega_{2}:=\frac{k-1}{2a}\omega.
$$
Then $(\omega_{1},\omega_{2}) \neq (0,0),$ 
$\omega_{1}-a\omega_{2}=\omega,$ 
$\omega_{1}+a\omega_{2}=k\omega,$ 
$c^{k}=c$ and a simple use of the estimate
$$
\Bigl| f(x+k\omega)-c^{k}f(x) \Bigr|\leq \sum_{j=1}^{k}| f(x+j\omega)-cf(x+(j-1)\omega) |,\quad x\in {\mathbb R}
$$
shows that the function $(x,t)\mapsto u(x,t),$ $(x,t)\in {\mathbb R}^{2}$ is $(S,{\mathbb D})$-asymptotically $(\omega,c)$-periodic. In the particular case $a=1$ and ${\mathbb D}:=\{(x,t) \in {\mathbb R}^{2} : x\geq 0,\ t\geq 0,\ x\geq t^{2}+1\},$ e.g., it suffices to assume that the restrictions of functions $f (\cdot)$ and $g^{[1]}(\cdot)$ to the interval $[0,\infty)$ are $S$-asymptotically $(\omega,c)$-periodic in the sense of \cite[Definition 3.1]{BIMV}.  

3. Concerning the composition theorems and applications to semilinear Cauchy problems, we may refer, e.g., to \cite[Theorem 3.5, Theorem 3.6]{y-k chang}, \cite[Theorem 3.3, Theorem 3.4]{BIMV} and \cite[Theorem 2.19]{multi-ce} for some results known in the one-dimensional setting so far. The proofs of all these results are very elementary and we will only reconsider here the semilinear Hammerstein
integral equation of convolution type on ${\mathbb R}^{n}$  \index{space!$SP({\mathbb R}^{n} :  X)$}
(cf. the last application of \cite[Section 3]{multi-ce}). By the foregoing, we know that the space $SP({\mathbb R}^{n} :  X)$ of all 
semi-periodic functions $F : {\mathbb R}^{n} \rightarrow X$ is convolution invariant (it is not a Banach space but only a complete metric space). Since
the composition principle \cite[Theorem 2.9.51]{nova-selected} admits a straightforward extension to the multi-dimensional setting, we are able to show that, under certain assumptions, the following semilinear Hammerstein
integral equation 
\begin{align}\label{multiintegralce}
y({\bf t})= \int_{{\mathbb R}^{n}}k({\bf t}-{\bf s})G({\bf s},y({\bf s}))\, d{\bf s},\quad {\bf t}\in {\mathbb R}^{n},
\end{align}
where $G : {\mathbb R}^{n} \times X  \rightarrow X$ is semi-$(c_{j},{\mathcal B})_{j\in {\mathbb N}_{n}}$-periodic with ${\mathcal B}$ being the collection of all bounded subsets of $X$ and $c_{j}=1$ for all $j\in {\mathbb N}_{n},$ has a unique semi-periodic solution.
Let us assume that 
there exists a finite real constant $L>0$ such that 
\begin{align*}
\bigl\|G({\bf t};y)-G\bigl({\bf t};y'\bigr)\bigr\|_{X} \leq L\bigl\| y-y'\bigr\|_{X} ,\quad {\bf t}\in {\mathbb R}^{n},\ y\in X,\ y'\in X.
\end{align*}
It can be simply shown that for any semi-periodic function $y : {\mathbb R}^{n} \rightarrow X$ we have that the mapping ${\bf t}\mapsto G({\bf t};y({\bf t})),$ ${\bf t}\in {\mathbb R}^{n}$ is semi-periodic, as well. Since the space of semi-periodic functions in ${\mathbb R}^{n}$ is convolution invariant, it follows that the mapping 
$$
SP({\mathbb R}^{n} :  X) \ni y \mapsto \int_{{\mathbb R}^{n}}k({\cdot}-{\bf s})G({\bf s},y({\bf s}))\, d{\bf s} \in SP({\mathbb R}^{n} :  X)
$$
is well defined. If we assume that $L\int_{{\mathbb R}^{n}}|k({\bf t})| \, d{\bf t}<1,$ then the use of
Banach contraction principle
yields that there exists a unique solution of \eqref{multiintegralce} which belongs to the space $SP({\mathbb R}^{n} :  X)$.

\end{document}